\documentclass[12pt,english]{amsart}
\usepackage{amsmath,amsthm,amssymb,amsfonts,amscd, array,latexsym}
\usepackage[noadjust]{cite}
\usepackage[T2A]{fontenc}
\usepackage[cp1251]{inputenc}
\usepackage[english]{babel}
\usepackage{diagmac2} 
\usepackage{tikz}
\usetikzlibrary{positioning,calc,arrows,arrows.meta}
\usepackage{bm}

\usepackage[pdftex, colorlinks, citecolor=red, unicode, linktocpage=false]{hyperref}

\usepackage{color} 

\newtheorem{theorem}{Theorem}[section]
\newtheorem{corollary}[theorem]{Corollary}
\newtheorem{lemma}[theorem]{Lemma}
\newtheorem{proposition}[theorem]{Proposition}
\newtheorem{conjecture}[theorem]{Conjecture}

\theoremstyle{definition}
\newtheorem{definition}[theorem]{Definition}
\newtheorem{example}[theorem]{Example}

\theoremstyle{remark}
\newtheorem{remark}[theorem]{Remark}
\numberwithin{equation}{section}

\makeatletter

\renewcommand{\p@enumii}{}

\makeatother

\newcommand{\RR}{\mathbb{R}}
\newcommand{\NN}{\mathbb{N}}

\def\<#1>{\langle #1 \rangle}

\newbox\onebox
\newcommand{\coherent}[1]{\mathbin{\setbox\onebox=\hbox{$=$}\lower0.5\ht%
\onebox\hbox{$\stackrel{#1}{=}$}}}

\newcommand{\tld}[1]{\tilde{#1}}
\newcommand{\ol}[1]{\overline{#1}}
\newcommand{\mbf}[1]{\mathbf{#1}}
\newcommand{\tldbf}[1]{\tilde{\mathbf{#1}}}

\newcommand{\ta}{\tilde{a}}
\newcommand{\tb}{\tilde{b}}
\newcommand{\tx}{\tilde{x}}
\newcommand{\ty}{\tilde{y}}
\newcommand{\tz}{\tilde{z}}
\newcommand{\tr}{\tilde{r}}

\newcommand{\tX}{\tilde{X}}

\newcommand{\pretan}[2]{
\def\a{}\def\b{#2}
\ifx\b\a \Omega_{\infty}^{#1}
\else \Omega_{\infty,\tilde #2}^{#1}\fi}

\newcommand{\bfpretan}[2]{
\def\a{}\def\b{#2}
\ifx\b\a \bm{\Omega}_{\infty}^{#1}
\else \bm{\Omega}_{\infty,\tilde #2}^{#1}\fi}

\newcommand{\sstable}[2]{\tilde #1_{\infty,\tilde #2}}
\newcommand{\seq}[3][\mathbb{N}]{(#2)_{#3\in #1}}

\newcommand{\sign}[1]{\operatorname{sign}(#1)}
\newcommand{\dist}[3][r]{d_{\tilde #1}(\tilde #2, \tilde #3)}
\newcommand{\Dist}[2][r]{\tilde{d}_{\tilde #1}(\tilde #2)}

\newcommand{\acr}{\newline\indent}

\begin{document}

\title[On equivalence of unbounded metric spaces at infinity]{On equivalence of unbounded metric spaces at infinity}

\author{Viktoriia Bilet}
\address{\textbf{Viktoriia Bilet}\acr
Institute of Applied Mathematics and Mechanics of NASU\acr
Dobrovolskogo str. 1, Slovyansk 84100, Ukraine}
\email{viktoriiabilet@gmail.com}

\author{Oleksiy Dovgoshey}
\address{\textbf{Oleksiy Dovgoshey}\acr
Institute of Applied Mathematics and Mechanics of NASU\acr
Dobrovolskogo str. 1, Slovyansk 84100, Ukraine}
\email{oleksiy.dovgoshey@gmail.com}

\subjclass[2020]{Primary 54E35.}
\keywords{Unbounded metric space, pseudometric space, equivalence relation, pretangent space at infinity.}

\begin{abstract}
Let \((X, d)\) be an unbounded metric space. To investigate the asymptotic behavior of \((X, d)\) at infinity, one can consider a sequence of rescaling metric spaces \((X, \frac{1}{r_n} d)\) generated by given sequence \((r_n)_{n \in \NN}\) of positive reals with \(r_n \to \infty\). Metric spaces that are limit points of the sequence \((X, \frac{1}{r_n} d)_{n \in \NN}\) will be called pretangent spaces to \((X, d)\) at infinity. We found the necessary and sufficient conditions under which two given unbounded subspaces of \((X, d)\) have the same pretangent spaces at infinity. In the case when \((X, d)\) is the real line with Euclidean metric, we also describe all unbounded subspaces of \((X, d)\) isometric to their pretangent spaces.
\end{abstract}

\maketitle

%
\tableofcontents
\clearpage

\section{Introduction}

Under asymptotic pretangent spaces to an unbounded metric space $(X,d)$ we mean the metric spaces which are the limits of rescaling metric spaces $\left(X, \frac{1}{r_n} d\right)$ for $r_n$ tending to infinity. The Gromov---Hausdorff convergence and the asymptotic cones are most often used for construction of such limits. Both of these approaches are based on higher-order abstractions (see, for example, \cite{Ro} for details), which makes them very powerful, but it does away the constructiveness. In this paper we use a more elementary, sequential approach, which was proposed in \cite{BD2017} and applied to description of finiteness, uniqueness, and completeness of asymptotic pretangent spaces in \cite{BD2018, BD2019, BD2019a} .

The sequential approach to defining the metric structure of \((X, d)\) at infinity can be informally described as performing the following operations:
\begin{itemize}
\item Among all sequences of points from \(X\), select those that tend to infinity with a given speed.
\item Determine a pseudometric \(D\) on the set of all such sequences so that \(D(\tx, \ty) = 0\) holds if and only if the sequences \(\tx\) and \(\ty\) are ``indistinguishable'' at infinity.
\item Transform the resulting pseudometric space into a metric space whose points are classes of indistinguishable at infinity sequences of points from \(X\).
\end{itemize}

The paper is organized as follows. 

In Section~\ref{sec2}, we recall basic facts related to general pseudometric spaces and introduce the concept of pseudoisometries. Some new interrelations between the closed and complete subsets of pseudometric spaces are described in Proposition~\ref{p2.7} and Corollaries~\ref{c2.8}--\ref{c2.9}. It was shown in Theorem~\ref{t2.10} that two pseudometric spaces admit a pseudoisometry if and only if their metric identifications are isometric.

In section~\ref{sec3} we define the notions of asymptotically pretangent and asymptotically tangent metric spaces, which are the main objects of our studies.

Some basic properties of asymptotically pretangent and tangent spaces are collected together in Section~\ref{sec4}. In particular, Proposition~\ref{p1.2.2} describes the sequences of points of an unbounded metric space \(X\) turning into a unique point \(\nu_{0}\) of asymptotically pretangent spaces such that \(\nu_0\) is ``close to \(X\) as much as possible''. 

The gluing pretangent and tangent spaces are considered in Section~\ref{sec5}. The main result of the section, Theorem~\ref{t5.2}, shows as porosity at infinity of a distance set of the original metric space \(X\) related to the porosity of a distance set of the gluing of pretangent spaces to \(X\).

In Section~\ref{sec6}, we investigate the asymptotically pretangent space to the metric subspaces of the real line \(\RR\). The main result of the section, Theorem~\ref{t4.3}, completely characterizes the subspaces of \(\RR\) which are isometric to their asymptotically pretangent spaces.

Conditions under which two subspaces of an unbounded metric spaces are asymptotically equivalent were found in Theorem~\ref{t3.1.6} of Section~\ref{sec7}.

In conclusion of this brief introduction, we note that there exist other techniques for studying metric spaces at infinity. As examples, we mention Gromov's products~\cite{BS, Sc}, balleans~\cite{PZ} and the Wijsman convergence \cite{LechLev, Wijs64, Wijs66, BDD2017, Beer1993, Bee1994}.

\section{About metrics and pseudometrics}
\label{sec2}

In what follows, we will denote by \(\NN\) the set of all strictly positive integer numbers, and \(\RR\) the set of all real numbers and write \(\RR_{+} = [0, \infty)\).

A \textit{metric} on a set \(X\) is a function \(d\colon X^{2} \to \RR\) such that for all \(x\), \(y\), \(z \in X\):
\begin{enumerate}
\item \(d(x,y) \geqslant 0\) with equality if and only if \(x=y\), the \emph{positivity property};
\item \(d(x,y)=d(y,x)\), the \emph{symmetry property};
\item \(d(x, y)\leq d(x, z) + d(z, y)\), the \emph{triangle inequality}.
\end{enumerate}

A \emph{metric space} is pair \((X, d)\) consisting of a set \(X\) and a metric \(d\) on \(X\). If \((X, d)\) is a metric space and \(A\) is a subset of \(X\), then, for simplicity, we well keep the notation \(A\) for the metric space \((A, d_{A \times A})\), where \(d_{A \times A}\) is the restriction of \(d \colon X \times X \to \RR\) on \(A \times A\).

A useful generalization of the concept of metric is the concept of pseudometric.

\begin{definition}\label{ch2:d2}
Let \(X\) be a set and let \(d \colon X^{2} \to \RR\) be a nonnegative, symmetric function such that \(d(x, x) = 0\) for every \(x \in X\). The function \(d\) is a \emph{pseudometric} on \(X\) if it satisfies the triangle inequality.
\end{definition}

If \(d\) is a pseudometric on \(X\), we say that \((X, d)\) is a \emph{pseudometric} \emph{space}. It is clear that a pseudometric \(d \colon X^{2} \to \RR\) is a metric if and only if \(d(x, y) = 0\) implies \(x = y\) for all \(x\), \(y \in X\).

Let \((X, d)\) be a pseudometric space. An \emph{open ball} with a \emph{radius} \(r > 0\) and a \emph{center} \(c \in X\) is the set 
\[
B_r(c) = \{x \in X \colon d(c, x) < r\}.
\]

A sequence \((x_n)_{n \in \mathbb{N}} \subseteq X\) is a \emph{Cauchy sequence} in \((X, d)\) if, for every \(r > 0\), there is an integer \(n_0 \in \NN\) such that \(x_n \in B_r(x_{n_0})\) for every \(n \geqslant n_0\). It is easy to see that \((x_n)_{n \in \mathbb{N}}\) is a Cauchy sequence if and only if 
\begin{equation*}
\lim_{n \to \infty} \sup\{d(x_n, x_{n+k}) \colon k \in \mathbb{N}\} = 0.
\end{equation*}

\begin{remark}\label{r2.17}
Here and later the symbol \((x_n)_{n\in \mathbb{N}} \subseteq X\) means that \(x_n \in X\) holds for every \(n \in \mathbb{N}\).
\end{remark}

A sequence \((x_n)_{n \in \mathbb{N}}\) of points in a pseudometric space \((X, d)\) is said to \emph{converge to} a point \(a \in X\),
\begin{equation*}
\lim_{n \to \infty} x_n = a,
\end{equation*}
if, for every open ball \(B\) containing \(a\), it is possible to find an integer \(n_0 \in \NN\) such that \(x_n \in B\) for every \(n \geqslant n_0\). Thus, \((x_n)_{n \in \mathbb{N}}\) is \emph{convergent} to \(a\) if and only if 
\begin{equation*}
\lim_{n \to \infty} d(x_n, a) = 0.
\end{equation*}
A sequence is convergent if it is convergent to some point. It is clear that every convergent sequence is a Cauchy sequence.

The next basic concept is the concept of \emph{completeness}.

\begin{definition}\label{d2.6}
A subset \(S\) of a pseudometric space \((X, d)\) is complete if for every Cauchy sequence \((x_n)_{n \in \mathbb{N}} \subseteq S\) there is a point \(a \in S\) such that 
\[
\lim_{n \to \infty} x_n = a.
\]
\end{definition}

As in the case of metric spaces, we can introduce a topology on pseudometric spaces using the open balls.

\begin{definition}\label{d2.4}
Let \((X, d)\) be a pseudometric space, \(\mbf{B}_X = \mbf{B}_{X, d}\) be the set of all open balls in \((X, d)\) and \(\tau\) be a topology on \(X\). We will say that \(\tau\) is generated by pseudometric \(d\) if \(\mbf{B}_X\) is an open base for \(\tau\). 
\end{definition}

Thus, \(\tau\) is generated by \(d\) if and only if every \(B \in \mbf{B}_X\) belongs to \(\tau\) and every nonempty \(A \in \tau\) is the union of some family of elements of \(\mbf{B}_X\).

In what follows, we assume that every pseudometric space \((X, d)\) is also a topological space endowed with the topology \(\tau\) generated by \(d\).

\begin{proposition}\label{p2.4}
If \(X\) is a complete metric space, then \(X\) is closed in every metric superspace of \(X\).
\end{proposition}

(See, for example, Theorem~10.2.1 \cite{Sea2007}).

This result becomes invalid in the case of pseudometrics. One can prove that a complete pseudometric space \(X\) is closed in all pseudometric superspaces of \(X\) if and only if \(X = \varnothing\). A valid pseudometric analogy of Proposition~\ref{p2.4} will be given below in Proposition~\ref{p2.7} and Corollary~\ref{c2.8}.

Let \(X\) be a set. A \emph{binary relation} on \(X\) is a subset of the Cartesian square
\[
X^2 = X \times X = \{\<x, y>\colon x, y \in X\}.
\]
A binary relation \(R \subseteq X^{2}\) is an \emph{equivalence relation} on \(X\) if the following conditions hold for all \(x\), \(y\), \(z \in X\):
\begin{enumerate}
\item \(\<x, x> \in R\), the \emph{reflexivity} property;
\item \((\<x, y> \in R) \Leftrightarrow (\<y, x> \in R)\), the \emph{symmetry} property;
\item \(((\<x, y> \in R) \text{ and } (\<y, z> \in R)) \Rightarrow (\<x, z> \in R)\), the \emph{transitivity} property.
\end{enumerate}

If \(R\) is an equivalence relation on \(X\), then an \emph{equivalence class} is a subset \([a]_R\) of \(X\) having the form
\begin{equation}\label{e1.1}
[a]_R = \{x \in X \colon \<x, a> \in R\}, \quad a \in X.
\end{equation}
The \emph{quotient set} of \(X\) with respect to \(R\) is the set of all equivalence classes \([a]_R\), \(a \in X\).

There exists the well-known, one-to-one correspondence between the equivalence relations and partitions.

Let \(X\) be a nonempty set and \(P = \{X_j \colon j \in J\}\) be a set of nonempty subsets of \(X\). Then \(P\) is a \emph{partition} of \(X\) with the blocks \(X_j\) if
\[
\bigcup_{j \in J} X_j = X,
\]
and \(X_{j_1} \cap X_{j_2} = \varnothing\) holds for all distinct \(j_1\), \(j_2 \in J\). 

\begin{proposition}\label{p2.2}
Let \(X\) be a nonempty set. If \(P = \{X_j \colon j \in J\}\) is a partition of \(X\) and \(R_P\) is a binary relation on \(X\) defined as
\begin{itemize}
\item[] \(\<x, y> \in R_P\) if and only if \(\exists j \in J\) such that  \(x \in X_j\) and \(y \in X_j\),
\end{itemize}
then \(R_P\) is an equivalence relation on \(X\) with the equivalence classes \(X_j\). Conversely, if \(R\) is an equivalence relation on \(X\), then the quotient set of \(X\) with respect to \(R\) is a partition of \(X\) with the blocks \([a]_R\).
\end{proposition}

For the proof, see, for example, \cite[Chapter~II, \S{}~5]{KurMost}.

For every pseudometric space \((X, d)\), we define a binary relation \(\coherent{0}\) on \(X\) as
\begin{equation}\label{ch2:p1:e1}
(x \coherent{0} y) \Leftrightarrow (d(x, y) = 0), \quad x, y \in X
\end{equation}
and, similarly~\eqref{e1.1}, we write 
\begin{equation}\label{e2.2}
[a]_0 = \{x \in X \colon d(a, x) = 0\},
\end{equation}
for every \(a \in X\).

\begin{proposition}\label{ch2:p1}
Let \(X\) be a nonempty set and let \(d \colon X^{2} \to \RR\) be a pseudometric on \(X\). Then \({\coherent{0}}\) is an equivalence relation on \(X\) and the function \(\delta_d\),
\begin{equation}\label{e1.1.5}
\delta_d(\alpha, \beta) := d(x, y), \quad x \in \alpha \in X/{\coherent{0}}, \quad y \in \beta \in X/{\coherent{0}},
\end{equation}
is a correctly defined metric on \(X/{\coherent{0}}\), where \(X/{\coherent{0}}\) is the quotient set of \(X\) with respect to \(\coherent{0}\).
\end{proposition}

The proof of Proposition~\ref{ch2:p1} can be found in \cite[Ch.~4, Th.~15]{Kelley1975}.
\medskip

In what follows we will sometimes say that \((X / \coherent{0}, \delta_d)\) is the \emph{metric identification} of the pseudometric space \((X, d)\).

Now we are ready to formulate and prove a partial analog of Proposition~\ref{p2.4} for the pseudometric spaces.

\begin{proposition}\label{p2.7}
Let \((X, d)\) be a pseudometric space and let \(A\) be a complete subset of \((X, d)\). Then \(A\) is closed in \((X, d)\) if and only if 
\begin{equation}\label{p2.7:e1}
[a]_0 \subseteq A
\end{equation}
holds for every \(a \in A\), where \([a]_0\) is defined by \eqref{e2.2}.
\end{proposition}

\begin{proof}
Let \(A\) be closed in \((X, d)\). Then the set \(X \setminus A\) is open in \((X, d)\) and, consequently, there is \(\mbf{B}^1 \subseteq \mbf{B}_X\) such that
\begin{equation}\label{p2.7:e2}
X \setminus A = \bigcup_{B \in \mbf{B}^1} B.
\end{equation}
Using Proposition~\ref{p2.2}, we see that \eqref{p2.7:e1} holds for every \(a \in A\) if and only if 
\begin{equation}\label{p2.7:e3}
[b]_0 \subseteq X \setminus A
\end{equation}
holds for every \(b \in X \setminus A\). From \eqref{p2.7:e2} it follows that \eqref{p2.7:e3} holds for every \(b \in X \setminus A\) if
\begin{equation}\label{p2.7:e4}
[b]_0 \subseteq B
\end{equation}
holds for every \(b \in B\) and every \(B \in \mbf{B}_X\). Now \eqref{p2.7:e4} can be directly proved by using the definition of open balls and the triangle inequality.

Conversely, let \eqref{p2.7:e1} hold for every \(a \in A\). If \(A\) is not a closed set, then \(X \setminus A\) is not an open set and, consequently, there is \(c \in X \setminus A\) such that
\[
B_r(c) \cap A \neq \varnothing
\]
for every \(r > 0\), where \(B_r(c)\) is the open ball with the center \(c\) and the radius \(r\). Consequently, there is a convergent sequence \(\seq{a_n}{n} \subseteq A\),
\begin{equation}\label{p2.7:e5}
\lim_{n \to \infty} d(a_n, c) = 0.
\end{equation}
Since \(A\) is complete and every convergent sequence is a Cauchy sequence, there is \(a \in A\) such that 
\begin{equation}\label{p2.7:e6}
\lim_{n \to \infty} d(a_n, a) = 0.
\end{equation}
Using the triangle inequality and \eqref{p2.7:e5}--\eqref{p2.7:e6}, we obtain
\[
d(c, a) \leqslant \limsup_{n \to \infty} d(c, a_n) + \limsup_{n \to \infty} d(a_n, a) = 0.
\]
Therefore, \(c \in [a_0]\) holds, that implies \(c \in A\) by \eqref{p2.7:e1}. Thus, we obtain the contradiction \(c \in A \cap (X \setminus A) = \varnothing\).
\end{proof}

\begin{corollary}\label{c2.8}
Let \((X, d)\) be a pseudometric space and let \(A\) be a complete subset of \(X\). Write \(\overline{A}\) for the closure of \(A\). Then \(\overline{A}\) is also a complete subset of \(X\) and the equality 
\begin{equation}\label{c2.8:e0}
\overline{A} = \bigcup_{a \in A} [a]_0
\end{equation}
holds.
\end{corollary}

\begin{proof}
First we note that the set 
\[
\bigcup_{a \in A} [a]_0
\]
is complete. Indeed, if \(\seq{x_n}{n} \subseteq \bigcup_{a \in A} [a]_0\) is a Cauchy sequence, then there is a sequence \(\seq{a_n}{n} \subseteq A\) such that 
\begin{equation}\label{c2.8:e1}
d(x_n, a_n) = 0
\end{equation}
for every \(n \in \NN\). The last equality and the definition of Cauchy sequences imply that \(\seq{a_n}{n}\) is also a Cauchy sequence. Since \(A\) is complete, there exists \(b \in A\) such that
\begin{equation}\label{c2.8:e2}
\lim_{n \to \infty} d(a_n, b) = 0.
\end{equation}
Using \eqref{c2.8:e1} and \eqref{c2.8:e2}, we see that 
\[
\lim_{n \to \infty} d(x_n, b) = 0,
\]
in addition, from 
\begin{equation}\label{c2.8:e3}
A \subseteq \bigcup_{a \in A} [a]_0
\end{equation}
it follows that \(b \in \bigcup_{a \in A} [a]_0\). Thus, \(\seq{x_n}{n}\) is convergent in \(\bigcup_{a \in A} [a]_0\) and, consequently, \(\bigcup_{a \in A} [a]_0\) is complete as required. 

Let us prove~\eqref{c2.8:e0}. Since the equality \([x]_0 = [a]_0\) holds for every \(a \in A\) and every \(x \in [a]_0\), we obtain that the set \(\bigcup_{a \in A} [a]_0\) is closed by Proposition~\ref{p2.7}. The last statement and \eqref{c2.8:e3} imply the inclusion
\[
\overline{A} \subseteq \bigcup_{a \in A} [a]_0,
\]
because every closed superset of \(A\) contains the closure of \(A\). To complete the proof, we recall that 
\[
S = \bigcup_{s \in S} [s]_0
\]
holds for every closed \(S \subseteq X\) (as was shown in the first part of the proof of Proposition~\ref{p2.7}) and, consequently, 
\[
\overline{A} = \bigcup_{b \in \overline{A}} [b]_0 \supseteq \bigcup_{a \in A} [a]_0,
\]
because \(\overline{A} \supseteq A\).
\end{proof}

Corollary~\ref{c2.8} implies the following.

\begin{corollary}\label{c2.9}
Let \((X, d)\) be a pseudometric space and let \(A_1\), \(A_2\) be subsets of \(X\). If the statement:
\begin{enumerate}
\item \label{c2.9:s1} For all \(x_1 \in A_1\) and \(x_2 \in A_2\), there are \(y_1 \in A_1\) and \(y_2 \in A_2\) such that 
\[
d(x_1, y_2) = d(x_2, y_1) = 0
\]
\end{enumerate}
is valid, then the equality
\begin{equation}\label{c2.9:e1}
\overline{A}_1 = \overline{A}_2
\end{equation}
holds. Moreover, if \(A_1\) and \(A_2\) are complete subsets of \(X\), then equality \eqref{c2.9:e1} implies the validity of \ref{c2.9:s1}.
\end{corollary}

\begin{proof}
Let \ref{c2.9:s1} hold. Then the equality \eqref{c2.9:e1} follows directly from Definition~\ref{d2.4}. If \(A_1\) and \(A_2\) are complete and \eqref{c2.9:e1} holds, then the validity of \ref{c2.9:s1} follows from \eqref{e2.2} and \eqref{c2.8:e0}.
\end{proof}

\begin{definition}\label{d2.11}
Let \((X, d)\) and \((Y, \rho)\) be pseudometric spaces. A mapping \(\Phi \colon X \to Y\) is \emph{distance preserving} if
\[
\rho(\Phi(x), \Phi(y)) = d(x, y)
\]
holds for all \(x\), \(y \in X\). 
\end{definition}

If \((X, d)\) and \((Y, \rho)\) are metric spaces, then the distance preserving mappings \(\Phi \colon X \to Y\) are usually called \emph{isometric embeddings} of \(X\) in \(Y\). If \(\Phi\) is a bijective isometric embedding, then we call this mapping an \emph{isometry} of \((X, d)\) and \((Y, \rho)\). Two metric spaces are \emph{isometric} if there is an isometry of these spaces.

The concept of isometry can be extended to pseudometric spaces in various non-equivalent ways. We will now give one of these generalizations that seems most appropriate for our purpose.

\begin{definition}\label{d2.5}
Let \((X, d)\) and \((Y, \rho)\) be pseudometric spaces. A mapping \(\Phi \colon X \to Y\) is a \emph{pseudoisometry} of \((X, d)\) and \((Y, \rho)\) if:
\begin{enumerate}
\item \label{d2.11:s1} \(\Phi\) is distance preserving;
\item \label{d2.11:s2} for every \(u \in Y\) there is \(v \in X\) such that \(\rho(\Phi(v), u) = 0\).
\end{enumerate}
\end{definition}

\begin{example}\label{ex2.3}
Let \(X\) be a set, \((Y, \rho)\) be a metric space and let \(\Phi \colon X \to Y\) be a surjective mapping. Then the function \(d \colon X^2 \to \RR\), defined by 
\[
d(x, y) := \rho(\Phi(x), \Phi(y)), \quad x, y \in X, 
\]
is a pseudometric on \(X\) and the mapping \(\Phi\) is a pseudoisometry of the pseudometric space \((X, d)\) and the metric space \((Y, \rho)\).
\end{example}

\begin{example}\label{ex2.5}
Let \((Y, \rho)\) be a pseudometric space and let \((Y / \coherent{0}, \delta_{\rho})\) be the metric identification of \((Y, \rho)\). Then the natural projection \(\pi_Y \colon Y \to Y / \coherent{0}\),
\begin{equation}\label{ex2.5:e1}
\pi_Y(y) = \{x \in Y \colon \rho(x, y) = 0\}, \quad y \in Y,
\end{equation}
is a pseudoisometry of the pseudometric space \((Y, \rho)\) and the metric space \((Y / \coherent{0}, \delta_{\rho})\).
\end{example}

\begin{example}\label{ex2.6}
Let \((Y, \rho)\), \((Y / \coherent{0}, \delta_{\rho})\) and \(\pi_Y\) be defined as in Example~\ref{ex2.5}. Then the Axiom of Choice implies the existence of a function \(\Psi \colon Y / \coherent{0} \to Y\) such that \(\pi_Y(\Psi(\alpha)) = \alpha\) for every \(\alpha \in Y / \coherent{0}\). Using Definition~\ref{d2.5} and equalities~\eqref{ch2:p1:e1}--\eqref{ex2.5:e1} (with \(\delta_d = \delta_{\rho}\) and \((X, d) = (Y, \rho)\)) we can easily show that \(\Psi\) is a pseudoisometry of \((Y / \coherent{0}, \delta_{\rho})\) and \((Y, \rho)\).
\end{example}

The following properties of pseudoisometries are easy to prove.

\begin{proposition}\label{p2.5}
Let \((X, d)\) and \((Y, \rho)\) be pseudometric spaces. Then the following statements hold for every pseudoisometry \(\Phi \colon X \to Y\):
\begin{enumerate}
\item \label{p2.5:s1} If \((X, d)\) and \((Y, \rho)\) are metric spaces, then \(\Phi\) is an isometry.
\item \label{p2.5:s2} If \((Y, \rho)\) is a metric space, then \(\Phi\) is a surjection.
\item \label{p2.5:s3} If \((X, d)\) is a metric space, then \(\Phi\) is an injection.
\item \label{p2.5:s4} \(\Phi\) is a homeomorphism of the topological spaces \(X\) and \(Y\) (endowed with topologies generated by \(d\) and \(\rho\), respectively) if and only if \(\Phi\) is bijective. 
\end{enumerate}
\end{proposition}

\begin{lemma}\label{l2.8}
Let \((X, d)\), \((Y, \rho)\) and \((Z, \delta)\) be pseudometric spaces and let \(\Phi \colon X \to Y\) and \(\Psi \colon Y \to Z\) be pseudoisometries. Then the mapping
\[
X \xrightarrow{\Phi} Y \xrightarrow{\Psi} Z
\]
is also a pseudoisometry.
\end{lemma}

\begin{proof}
It follows directly from Definition~\ref{d2.5}.
\end{proof}

We say that two pseudometric spaces are \emph{pseudoisometric} if there is a pseudoisometry of these spaces.

\begin{theorem}\label{t2.10}
Let \((X, d)\) and \((Y, \rho)\) be pseudometric spaces. Then \((X, d)\) and \((Y, \rho)\) are pseudoisometric if and only if the metric identifications \((X / \coherent{0}, \delta_d)\) and \((Y / \coherent{0}, \delta_\rho)\) are isometric metric spaces.
\end{theorem}

\begin{proof}
Suppose that \((X, d)\) and \((Y, \rho)\) are pseudoisometric. We will prove that \((X / \coherent{0}, \delta_d)\) and \((Y / \coherent{0}, \delta_{\rho})\) are isometric. 

Write \(\Phi \colon X \to Y\) for a pseudoisometry of \((X, d)\) and \((Y, \rho)\). Let \(\pi_X \colon X \to X / \coherent{0}\) and \(\pi_Y \colon Y \to Y / \coherent{0}\) be the natural projection of \(X\) and \(Y\) on the quotient sets \(X / \coherent{0}\) and \(Y / \coherent{0}\), respectively, i.e.,
\begin{align}\label{p2.6:e1}
\pi_X (x) &= \{z \in X \colon d(x, z) = 0\},\\
\label{p2.6:e2}
\pi_Y (y) &= \{t \in Y \colon \rho(t, y) = 0\}
\end{align}
hold for all \(x \in X\) and \(y \in Y\). 

We claim that there is a mapping \(F \colon X / \coherent{0} \to Y / \coherent{0}\) such that the diagram
\begin{equation}\label{p2.6:e3}
\ctdiagram{
\ctv 0,30: {X} 
\ctv 120,30: {Y}
\ctv 0,-30: {X / \coherent{0}} 
\ctv 120,-30: {Y / \coherent{0}}
\ctet 0,30,120,30:{\Phi}
\ctet 0,-30,120,-30:{F}
\ctel 0,30, 0,-30:{\pi_X}
\cter 120,30, 120,-30:{\pi_Y}
}
\end{equation}
is commutative, i.e., \(F \circ \pi_X = \pi_Y \circ \Phi\) holds. A desired \(F\) can be found if and only if the implication
\begin{equation}\label{p2.6:e4}
\bigl(\pi_X(x_1) = \pi_X(x_2)\bigr) \Rightarrow \bigl(\pi_Y(\Phi(x_1)) = \pi_Y(\Phi(x_2))\bigr)
\end{equation}
is valid for all \(x_1\), \(x_2 \in X\). Let us prove the validity of \eqref{p2.6:e4}. If we have \(\pi_X(x_1) = \pi_X(x_2)\), then \eqref{p2.6:e1} implies \(d(x_1, x_2) = 0\). The last equality and statement \ref{d2.11:s1} from Definition~\ref{d2.5} give us the equality
\begin{equation}\label{p2.6:e5}
\rho(\Phi(x_1), \Phi(x_2)) = 0.
\end{equation}
Now from \eqref{p2.6:e5} and \eqref{p2.6:e2} it follows that \(\pi_Y(\Phi(x_1)) = \pi_Y(\Phi(x_2))\). Thus, there is \(F \colon X / \coherent{0} \to Y / \coherent{0}\) such that diagram \eqref{p2.6:e3} is commutative.

Let \(F \colon X / \coherent{0} \to Y / \coherent{0}\) satisfy 
\begin{equation}\label{p2.6:e10}
F \circ \pi_X = \pi_Y \circ \Phi
\end{equation}
and let \(\alpha\), \(\beta\) be arbitrary points of \(X / \coherent{0}\). Then, by \eqref{e1.1.5}, we have
\begin{equation}\label{p2.6:e6}
\delta_d (\alpha, \beta) = d(x_1, x_2)
\end{equation}
for all \(x_1 \in \alpha\) and \(x_2 \in \beta\). Using \eqref{p2.6:e1} we can rewrite \eqref{p2.6:e6} as
\begin{equation}\label{p2.6:e7}
\delta_d (\pi_X(x_1), \pi_X(x_2)) = d(x_1, x_2).
\end{equation}
The last equality and statement \ref{d2.11:s1} from Definition~\ref{d2.5} imply
\begin{equation}\label{p2.6:e8}
d(x_1, x_2) = \rho (\Phi(x_1), \Phi(x_2)).
\end{equation}
Now using Proposition~\ref{ch2:p1} with \(X = Y\) and \(d = \rho\), we obtain 
\begin{equation}\label{p2.6:e9}
\rho (\Phi(x_1), \Phi(x_2)) = \delta_{\rho}(\pi_Y(\Phi(x_1)), \pi_Y(\Phi(x_2))).
\end{equation}
Consequently, we have 
\begin{equation}\label{p2.6:e11}
\delta_d (\pi_X(x_1), \pi_X(x_2)) = \delta_{\rho}(\pi_Y(\Phi(x_1)), \pi_Y(\Phi(x_2))).
\end{equation}
Using \eqref{p2.6:e10} we can rewrite \eqref{p2.6:e11} as
\[
\delta_d (\pi_X(x_1), \pi_X(x_2)) = \delta_{\rho}(F(\pi_X(x_1)), F(\pi_X(x_2))).
\]
Thus, we have
\[
\delta_d (\alpha, \beta) = \delta_{\rho} (F(\alpha), F(\beta))
\]
for all \(\alpha\), \(\beta \in X / \coherent{0}\). Consequently, \(F\) is an isometric embedding of the metric space \((X/\coherent{0}, \delta_d)\) in the metric space \((Y/\coherent{0}, \delta_{\rho})\). Moreover, by Example~\ref{ex2.5}, the mapping \(\pi_Y \colon Y \to Y / \coherent{0}\) is a pseudoisometry and, consequently, 
\[
X \xrightarrow{\Phi} Y \xrightarrow{\pi_Y} Y / \coherent{0}
\]
is a pseudoisometry by Lemma~\ref{l2.8}. Since \((Y/\coherent{0}, \delta_{\rho})\) is a metric space, the pseudoisometry \(X \xrightarrow{\Phi} Y \xrightarrow{\pi_Y} Y / \coherent{0}\) is surjective by Statement~\ref{p2.5:s2} of Proposition~\ref{p2.5}. Using equality~\eqref{p2.6:e10}, we obtain that \(F \circ \pi_X\) is also surjective, that implies the surjectivity of \(F\). Every surjective isometric embedding is an isometry. Thus, \(F\) is an isometry of \((X/\coherent{0}, \delta_d)\) and \((Y/\coherent{0}, \delta_{\rho})\).

Conversely, let \((X/\coherent{0}, \delta_d)\) and \((Y/\coherent{0}, \delta_{\rho})\) be isometric metric spaces. Write \(F \colon X/\coherent{0} \to Y/\coherent{0}\) for an isometry of these spaces. Let \(\pi_X \colon X \to X/\coherent{0}\) be the natural projection defined by \eqref{p2.6:e1} and let \(\Psi \colon Y/\coherent{0} \to Y\) be defined as in Example~\ref{ex2.6}. Then, using the same example and Lemma~\ref{l2.8}, we see that the mapping 
\[
X \xrightarrow{\pi_X} X/\coherent{0} \xrightarrow{F} Y/\coherent{0} \xrightarrow{\Psi} Y
\]
is a pseudoisometry of \((X, d)\) and \((Y, \rho)\).
\end{proof}

\begin{corollary}\label{c2.7}
The relation ``to be pseudoisometric'' is an equivalence relation on the class of all pseudometric spaces.
\end{corollary}

\begin{proof}
The relation ``to be isometric'' is an equivalence relation on the class of all metric spaces.
\end{proof}

Analyzing the proof of Theorem~\ref{t2.10}, we obtain the following.

\begin{corollary}\label{c2.13}
Let \(\Phi \colon X \to Y\) be a pseudoisometry of pseudometric spaces \((X, d)\) and \((Y, \rho)\). If \(F \colon X / {\coherent{0}} \to Y / {\coherent{0}}\) is a mapping such that diagram~\eqref{p2.6:e3} is commutative, then \(F\) is an isometry of the metric spaces \((X / {\coherent{0}}, \delta_d)\) and \((Y / {\coherent{0}}, \delta_{\rho})\).
\end{corollary}

\begin{remark}\label{r2.11}
John Kelley \cite{Kelley1975} define the isometries of pseudometric spaces \(X\) and \(Y\) as the distance preserving surjective mappings \(X \to Y\). This concept is a generalization of the concept of the usual isometries of metric spaces and it is clear that every isometry of pseudometric spaces in Kelley's sense is a pseudoisometry of these spaces. Another generalization of the concept of isometry to pseudoisometric spaces was given in~\cite{DLAMH2020, DovBBMSSS2020}.
\end{remark}

\section{Asymptotically pretangent and tangent spaces. Initial definitions}
\label{sec3}

Let $\tr=(r_n)_{n\in\NN}$ be a sequence of positive real numbers with $\lim_{n\to\infty} r_{n} = \infty$. In what follows $\tr$ will be called a \emph{scaling sequence}. Moreover, the formula $\seq{x_n}{n} \subset A$ will be mean that all elements of the sequence $\seq{x_n}{n}$ belong to the set~$A$.

\begin{definition}\label{d1.1.1}
Let \((X, d)\) be an unbounded metric space. Two sequences $\tx = \seq{x_n}{n} \subset X$ and $\ty = \seq{y_n}{n}\subset X$ are \emph{mutually stable} with respect to a scaling sequence $\tr=\seq{r_n}{n}$ if there is a finite limit
\begin{equation}\label{e1.1.1}
\lim_{n\to\infty}\frac{d(x_n,y_n)}{r_n} := d_{\tr}(\tx,\ty).
\end{equation}
\end{definition}

If \((X, d)\) is an unbounded metric space and \(\tr\) is a scaling sequence, then we denote by $Seq(X, \tr)$ the set of all sequences $\tx = \seq{x_n}{n} \subset X$ for which there is a finite limit
\begin{equation}\label{e1.1.2}
\lim_{n\to\infty} \frac{d(x_n, p)}{r_n} := \Dist{x},
\end{equation}
where \(p\) is a fixed point of \(X\).

\begin{remark}\label{r1.1.3}
The triangle inequality implies that the set $Seq(X, \tr)$ is invariant with respect to the choosing $p\in X$ in \eqref{e1.1.2}.
\end{remark}

\begin{definition}\label{d1.1.2}
Let \((X, d)\) be an unbounded metric space and \(\tr\) be a scaling sequence. A set $F\subseteq Seq(X, \tr)$ is \emph{self-stable} if any two $\tx, \ty \in F$ are mutually stable. $F$ is \emph{maximal self-stable} if it is self-stable and, for arbitrary $\ty\in Seq(X, \tr)$, we have either $\ty\in F$ or there is $\tx\in F$ such that $\tx$ and $\ty$ are not mutually stable.
\end{definition}

The maximal self-stable subsets of $Seq(X, \tr)$ will be denoted by $\sstable{X}{r}$. By Zorn's lemma, for every self-stable \(A \subseteq Seq(X, \tr)\), there is \(\sstable{X}{r}\) such that \(A \subseteq \sstable{X}{r}\).

Let \(\sstable{X}{r}\) be a maximal self-stable subset of \(Seq(X, \tr)\) and let $d_{\tr} \colon \sstable{X}{r} \times \sstable{X}{r} \rightarrow \RR$ be defined by \eqref{e1.1.1}. Obviously, $d_{\tr}$ is a symmetric non-negative function and \(d_{\tr} (\tx, \tx) = 0\) holds for every \(\tx \in \tX_{\infty, \tr}\). Moreover, the triangle inequality for $d$ gives us the triangle inequality for \(d_{\tr}\),
$$
d_{\tr}(\tx,\ty) \leq d_{\tr} (\tx,\tz) + d_{\tr} (\tz,\ty)
$$
whenever \(\tx\), \(\ty\), \(\tz \in \sstable{X}{r}\). Hence, $(\sstable{X}{r},d_{\tr})$ is a pseudometric space.

Now we are ready to introduce the main object of our research.

\begin{definition}\label{d1.1.4}
Let $(X,d)$ be an unbounded metric space and let \(\tr\) be a scaling sequence. The metric identifications of pseudometric spaces \((\sstable{X}{r}, d_{\tr})\) will be denoted by \((\pretan{X}{r}, \rho)\) and called the \emph{asymptotically pretangent spaces} to \((X, d)\) with respect to \(\tr\).

A metric space \((Y, \Delta)\) is \emph{asymptotically pretangent} to \((X, d)\) if there are \(\tr\) and \(\pretan{X}{r}\) such that \((Y, \Delta) = (\pretan{X}{r}, \rho)\).
\end{definition}

Let \(\tr = (r_n)_{n \in \NN}\) be a scaling sequence. As usual, a subsequence \((b_k)_{k \in \NN}\) of \(\tr\) is a sequence defined by \(b_k = r_{n_k}\), where \(n_1 < n_2 < \ldots\) is a strictly increasing sequence of indexes. It is clear that every subsequence of a scaling sequence is also a scaling sequence.

Let $\tr' = \seq{r_{n_k}}{k}$ be a subsequence of a scaling sequence $\tr$. Then, for every $\tx=\seq{x_n}{n}\in Seq(X, \tr)$, we write $\tx' := (x_{n_k})_{k\in\NN}$. It is easy to see that we have
\[
\{\tx' \colon \tx\in Seq(X, \tr)\}\subseteq Seq(X, \tr')
\]
and
\begin{equation}\label{e3.3}
\Dist[r']{x'} = \Dist{x}
\end{equation}
for every $\tx\in Seq(X, \tr)$. Moreover, if sequences $\ty$, $\tz \in Seq(X, \tr)$ are mutually stable with respect to $\tr$, then $\ty'$ and $\tz'$ are mutually stable with respect to $\tr'$ and
\begin{equation}\label{e1.1.6}
d_{\tr}(\ty, \tz)=d_{\tr'}(\ty', \tz')
\end{equation}
holds. Consequently, by Zorn's lemma, for every $\sstable{X}{r}\subseteq Seq(X, \tr)$, there is $\tilde{X}_{\infty, \tr'} \subseteq Seq(X, \tr')$ such that
\begin{equation}\label{e1.1.7}
\{\tx':\tx \in \tilde{X}_{\infty,\tr}\}\subseteq \tilde{X}_{\infty,\tr'}.
\end{equation}
Let us denote by $\varphi_{\tr'}$ the mapping from $\tilde{X}_{\infty,\tr}$ to $\tilde{X}_{\infty,\tr'}$ defined by $\varphi_{\tr'}(\tx)=\tx'$ for all $\tx\in\tilde{X}_{\infty,\tr}$. It follows from~\eqref{e1.1.6} that, after corresponding metric identifications, the mapping $\varphi_{\tr'}\colon \sstable{X}{r} \to \tilde{X}_{\infty, \tr'}$ passes to an isometric embedding $em'\colon \pretan{X}{r} \rightarrow \pretan{X}{r'}$ such that the
diagram
\begin{equation}\label{e1.1.8}
\begin{CD}
\sstable{X}{r} @>\varphi_{\tr'}>> \tilde{X}_{\infty, \tr'}\\
@V{\pi}VV @VV{\pi'}V \\
\pretan{X}{r} @>em'>> \pretan{X}{r'}
\end{CD}
\end{equation}
is commutative. Here, for all $\tx\in \sstable{X}{r}$ and $\tilde t\in\tilde{X}_{\infty, \tr'}$,
\begin{equation}\label{e1.1.9}
\begin{aligned}
\pi(\tx) &:= \{\ty \in \tilde{X}_{\infty,\tr}\colon d_{\tr}(\tx, \ty)=0\},\\
\pi'(\tilde t) &:= \{\tz \in \tilde{X}_{\infty,\tr'} \colon d_{\tr'}(\tld{t}, \tz)=0\}.
\end{aligned}
\end{equation}

In what follows, for every scaling sequence \(\tr\), the set of all subsequences of \(\tr\) we will be denote by \(\tld{\mathbf{R}} = \tld{\mathbf{R}}(\tr)\).

\begin{definition}\label{d1.1.5}
Let $(X,d)$ be an unbounded metric space, $\tr$ be a scaling sequence, \(\sstable{X}{r}\) be a maximal self-stable subset of \(Seq(X, \tr)\) and let \(\pretan{X}{r}\) be the metric identification of \(\sstable{X}{r}\). The space $\pretan{X}{r}$ is \emph{asymptotically tangent} to \((X, d)\) with respect to \(\tr\) if, for every \(\tr' \in \tldbf{R}\), there is \(\sstable{X}{r'}\) satisfying \eqref{e1.1.7} such that, for the metric identification \(\pretan{X}{r'}\) of \(\sstable{X}{r'}\), the isometric embedding $em'\colon \pretan{X}{r} \rightarrow \pretan{X}{r'}$ is an isometry.
\end{definition}

\begin{remark}\label{r2.8}
A set of reformulations of Definition~\ref{d1.1.5} will be given in the next section of the paper (see Proposition~\ref{p1.3.1}). Now we only note that \(em' \colon \pretan{X}{r} \to \pretan{X}{r'}\) is an isometry if and only if it is a surjection.
\end{remark}

\section{Elementary properties}
\label{sec4}

Let \((X, d)\) be an unbounded metric space. It is clear that every bounded \(\seq{x_n}{n} \subset X\) belongs to $Seq(X, \tr)$ for every scaling sequence \(\tr\). Let us consider now the set \(\tilde{X}_{\infty}\) of all sequences $\seq{x_n}{n}\subset X$ satisfying the limit relation $\lim_{n\to\infty} d(x_n, p) = \infty$, where \(p\) is a given point of \(X\). 

\begin{proposition}\label{p1.2.1}
Let $(X, d)$ be an unbounded metric space. Then the following statements hold.
\begin{enumerate}
\item\label{p1.2.1s1} The intersection $\tld{X}_{\infty} \cap Seq(X, \tr)$ is nonempty for every scaling sequence $\tr.$
\item\label{p1.2.1s2} For every $\tx \in\tilde{X}_{\infty},$ there exists a scaling sequence $\tr$ such that $\tx\in Seq(X, \tr)$ and $\Dist{x} = 1$.
\end{enumerate}
\end{proposition}

\begin{proof}
\ref{p1.2.1s1} Let $\tr=(r_n)_{n\in\NN}$ be a scaling sequence and let $p\in X$. Let us denote by $\ol{B}(p, r_{n}^{1/2})$ the closed ball
\begin{equation}\label{e1.2.1}
\{x\in X: d(x, p)\le r_{n}^{1/2}\}.
\end{equation}
Write
\begin{equation}\label{e1.2.2}
k_{n} := \sup\{d(x, p): x\in \ol{B}(p, r_n^{1/2})\},
\end{equation}
$n= 1,2, \ldots$. We can find $\tx=\seq{x_n}{n}\subset X$ such that
\begin{equation}\label{e1.2.3}
\lim_{n\to\infty}\frac{k_n}{d(x_n, p)}=1.
\end{equation}
Since $X$ is unbounded and \(\lim_{n\to\infty} r_{n} = \infty\), the equality $\lim_{n\to\infty} k_{n} = \infty$ holds. Consequently, we have \(\lim_{n\to\infty} d(x_n, p)=\infty\), i.e., $\tx\in\tilde{X}_{\infty}$. It follows from \eqref{e1.2.1} and \eqref{e1.2.2} that $k_n\le r_{n}^{1/2}$ holds for every $n\in\NN.$ The last inequality and \eqref{e1.2.3} imply
\[
\lim_{n \to \infty} \frac{d(x_n, p)}{r_n}=0.
\]
Thus, \(\tx\) belongs to \(\tld{X}_{\infty} \cap Seq(X, \tr)\).

\ref{p1.2.1s2} Let $\tx = \seq{x_n}{n}\in\tilde{X}_{\infty}$ and let $p\in X$. Then we have $\lim_{n\to\infty} d(x_n, p) = \infty$. Define a sequence $\tr=(r_n)_{n\in\NN}$ as:
\begin{equation*}\label{Func}
r_{n}:=\begin{cases}
d(x_n, p) & \mbox{if } x_n \ne p\\
1 & \mbox{if } x_n = p.
\end{cases}
\end{equation*}
From $\tx \in \tilde{X}_{\infty}$ it follows that $\lim_{n\to\infty} r_{n} = \infty$. Hence, $\tr$ is a scaling sequence. It is clear that $\Dist{x}=1$. Thus, $Seq(X, \tr)$ contains $\tx$.
\end{proof}

For every unbounded metric space $(X, d)$ and every scaling sequence $\tr$, we define the set $\sstable{X}{r}^{0}$ of sequences \(\tz = \seq{z_n}{n} \subset X\) by the rule:

\begin{equation}\label{e1.2.4}
\left(\tz \in \sstable{X}{r}^{0}\right) \Leftrightarrow \left(\lim_{n\to\infty} \frac{d(z_n, p)}{r_n} = 0\right),
\end{equation}
where $p$ is a point of $X$.

\begin{remark}\label{r1.2.3}
The set $\sstable{X}{r}^{0}$ is invariant under replacing of $p\in X$ by arbitrary point of \(X\) (cf. Remark~\ref{r1.1.3}).
\end{remark}

In Proposition~\ref{p1.2.2} below, we collect together some basic properties of the set $\sstable{X}{r}^{0}$. Similar results were formulated without proofs at \cite{BD2018, BD2019}.

\begin{proposition}\label{p1.2.2}
Let $(X, d)$ be an unbounded metric space and let $\tr$ be a scaling sequence. Then the following statements hold.
\begin{enumerate}
\item\label{p1.2.2s0} We have \(\sstable{X}{r}^{0} \subseteq Seq(X, \tr)\) and \(\tx \in \sstable{X}{r}^{0}\) for every bounded \(\tx = \seq{x_n}{n} \subset X\).
\item\label{p1.2.2s1} The set $\sstable{X}{r}^{0} \cap \tilde{X}_{\infty}$ is nonempty.
\item\label{p1.2.2s2} If we have $\tz\in \sstable{X}{r}^{0}$, $\ty = \seq{y_n}{n} \subset X$ and $d_{\tr}(\tz, \ty)=0$, then $\ty\in\sstable{X}{r}^{0}$.
\item\label{p1.2.2s3} If $F$ is a self-stable subset of \(Seq(X, \tr)\), then $\sstable{X}{r}^{0}\cup F$ is also a self-stable subset of $Seq(X, \tr)$.
\item\label{p1.2.2s4} The set $\sstable{X}{r}^{0}$ is self-stable.
\item\label{p1.2.2s5} The inclusion $\sstable{X}{r}^{0}\subseteq\sstable{X}{r}$ holds for every maximal self-stable subset $\sstable{X}{r}$ of $Seq(X, \tr).$
\item\label{p1.2.2s6} Let $\tz \in \tilde{X}_{\infty, \tr}^{0}$ and $\tx = \seq{x_n}{n} \subset X$. Then $\tx\in Seq (X, \tr)$ if and only if $\tx$ and $\tz$ are mutually stable. 

\item\label{p1.2.2s7} For every $\tx\in Seq (X, \tr)$ and every \(\tz \in \tilde{X}_{\infty, \tr}^{0}\) we have 
$$
\Dist[r]{x}=d_{\tr}(\tx, \tz).
$$
\item\label{p1.2.2s8} The set \(\sstable{X}{r}^{0}\) is a point of every asymptotically pretangent space \(\pretan{X}{r}\).
\item\label{p1.2.2s9} The equality $\varphi_{\tr'}(\sstable{X^0}{r}) = \sstable{X^0}{r'}$ holds for every $\tr'$, where $\varphi_{\tr'}$ is defined as in~\eqref{e1.1.8}.
\end{enumerate}
\end{proposition}

\begin{proof}
In what follows, \(p\) will be denote a given point of \(X\).

\ref{p1.2.2s0} It follows directly from the definitions of \(Seq(X, \tr)\) and \(\sstable{X}{r}^{0}\).

\ref{p1.2.2s1} This statement follows from the proof of statement~\ref{p1.2.1s1} in Proposition~\ref{p1.2.1}.

\ref{p1.2.2s2} Let \(\tz \in \sstable{X}{r}^{0}\), \(\ty = \seq{y_n}{n} \subset X\) and \(d_{\tr}(\tz, \ty)=0\) hold. To prove $\ty \in \sstable{X}{r}^{0}$, we note that
\begin{equation*}
0 \le \limsup_{n\to\infty} \frac{d(y_n, p)}{r_n} \le d_{\tr}(\tz, \ty) + \Dist{z}=0.
\end{equation*}
It implies the existence of the finite limit 
\[
\Dist[r]{y} = \lim_{n\to\infty} \frac{d(y_n, p)}{r_n}
\]
and the equality \(\Dist{y} = 0\). Thus, \(\ty \in \sstable{X}{r}^{0}\) holds.

\ref{p1.2.2s3} Let $F\subseteq Seq(X, \tr)$ be self-stable. It is clear that $\Dist{y}$ exists for every $\ty \in F \cup \sstable{X}{r}^{0}$. Hence $F \cup \sstable{X}{r}^{0}$ is self-stable if and only if $\tz$ and $\tx$ are mutually stable for all $\tx, \tz\in F\cup\sstable{X}{r}^{0}$. If \(\tz\), \(\tx \in F\), then $\tz$ and $\tx$ are mutually stable because \(F\) is self-stable. 

Suppose $\tx\in F$ and $\tz\in\sstable{X}{r}^{0}$. The double inequality
\begin{equation*}
d(x_n, p) - d(z_n, p) \leq d(x_n, z_n) \leq d(x_n, p) + d(z_n, p),
\end{equation*}
the membership \(\tx \in Seq(X, \tr)\), and the equality
$$
\lim_{n\to\infty} \frac{d(z_n, p)}{r_n} = 0
$$
imply the existence of $d_{\tr}(\tx, \tz)$. 

The case $\tx, \tz\in\sstable{X}{r}^{0}$ is similar. Thus the set $F\cup\sstable{X}{r}^{0}$ is self-stable.

\ref{p1.2.2s4} This follows from~\ref{p1.2.2s3} with $F=\varnothing$.

\ref{p1.2.2s5} Using statement~\ref{p1.2.2s3} with $F=\sstable{X}{r}$ we see that $\sstable{X}{r}^{0} \cup \sstable{X}{r}$ is self-stable. Since $\sstable{X}{r}$ is maximal self-stable, the equality $\sstable{X}{r}^{0} \cup \sstable{X}{r} = \sstable{X}{r}$ holds. Thus, $\sstable{X}{r}^{0}\subseteq\sstable{X}{r}$.

\ref{p1.2.2s6} Let \(\tz \in\tilde{X}_{\infty, \tr}^{0}\) and \(\tx = \seq{x_n}{n} \subset X\). Suppose $\tx$ and $\tz$ are mutually stable. Then, using the triangle inequality and \eqref{e1.2.4}, we obtain
\begin{equation}\label{p4.2:e1}
\limsup_{n\to\infty}\frac{d(x_n, p)}{r_n} \le d_{\tr}(\tx, \tz) \le \liminf_{n\to\infty} \frac{d(x_n, p)}{r_n}.
\end{equation}
Hence, $\tx\in Seq(X, \tr)$ holds. 

If $\tx\in Seq(X, \tr)$, then, similarly to~\eqref{p4.2:e1}, we obtain
\begin{equation}\label{p1.2.2e2}
\limsup_{n\to\infty} \frac{d(x_n, z_n)}{r_n} \le \Dist{x} \le \liminf_{n\to\infty} \frac{d(x_n, z_n)}{r_n}.
\end{equation}
Consequently, there is a finite limit
\[
\lim_{n \to \infty} \frac{d(x_n, z_n)}{r_n},
\]
i.e., $\tx$ and $\tz$ are mutually stable.

\ref{p1.2.2s7} Let \(\tx \in Seq (X, \tr)\) and \(\tz \in \tilde{X}_{\infty, \tr}^{0}\). Then the equality \(\Dist{x} = d_{\tr} (\tx, \tz)\) follows from \eqref{p1.2.2e2}.

\ref{p1.2.2s8} It follows from~\ref{p1.2.2s2}, \ref{p1.2.2s5} and the definition of pretangent spaces.

\ref{p1.2.2s9} Let \((n_k)_{k \in \NN}\) be a strictly increasing sequence of natural numbers, $\tr' := \seq{r_{n_k}}{k}$, $\ty = \seq{y_k}{k} \in \sstable{X^0}{r'}$ and let $\tx = \seq{x_n}{n} \in \sstable{X^0}{r}$. Write
$$
z_n := \begin{cases}
y_k & \text{if $n = n_k$ for some \(k \in \NN\)}\\
x_n & \text{otherwise}.
\end{cases}
$$
Then, for $\tz = \seq{z_n}{n}$, we have $\tz' = \ty$ and $\tz \in \sstable{X^0}{r}$.
\end{proof}

The following lemma is modification of Lemma~2.1 from \cite{BD2018}.

\begin{lemma}\label{l1.2.4}
Let $(X, d)$ be an unbounded metric space, $\ty$ be a sequence of points of \(X\), $\tr$ be a scaling sequence and let $\sstable{X}{r}$ be a maximal self-stable subset of \(Seq(X, \tr)\). If $\ty$ and $\tx$ are mutually stable for every $\tx\in\sstable{X}{r}$, then $\ty\in\sstable{X}{r}$.
\end{lemma}

\begin{proof}
Suppose that \(\ty\) and \(\tx\) are mutually stable for every \(\tx \in \sstable{X}{r}\). Let us consider an arbitrary \(\tz \in \sstable{X}{r}^{0}\). By statement~\ref{p1.2.2s5} of Proposition~\ref{p1.2.2}, we obtain
\begin{equation}\label{l1.2.4:e1}
\tz \in \sstable{X}{r}.
\end{equation}
Since \(\ty\) and \(\tz\) are mutually stable by supposition, \(\ty \in Seq(X, \tr)\) holds by statement \ref{p1.2.2s6} of that proposition. Using our supposition again, we obtain \(\ty \in \sstable{X}{r}\), because \(\sstable{X}{r}\) is a maximal self-stable subset of \(Seq(X, \tr)\).
\end{proof}

\begin{lemma}\label{l1.2.5}
Let $(X, d)$ be an unbounded metric space, let $\tr$ be a scaling sequence and let \(\tx = \seq{x_n}{n}\), \(\ty = \seq{y_n}{n}\) and \(\tld{t} = \seq{t_n}{n}\) be sequences of points of \(X\). If $\tx$ and $\ty$ are mutually stable with respect to $\tr$, and $d_{\tr}(\tx, \tilde t)=0$, then $\ty$ and $\tilde t$ are mutually stable with respect to $\tr$.
\end{lemma}

\begin{proof}
It follows from \eqref{e1.1.1}, the equality $d_{\tr}(\tx, \tilde t)=0$ and the inequalities
\begin{align*}
d_{\tr}(\tx, \ty) - d_{\tr}(\tx, \tilde t) & \leq \liminf_{n\to\infty} \frac{d(y_n,t_n)}{r_n} \\
&\leq \limsup_{n\to\infty} \frac{d(y_n,t_n)}{r_n} \leq d_{\tr}(\tx, \ty) + d_{\tr}(\tx, \tilde t). \qedhere
\end{align*}
\end{proof}

\begin{lemma}\label{l4.6}
Let \((X, d)\) be an unbounded metric space, \(\tr\) be a scaling sequence, let \(\tx = \seq{x_n}{n} \in Seq(X, \tr)\) and let \(\tld{t} = \seq{t_n}{n} \subset X\). If \(\tx\) and \(\tld{t}\) are mutually stable with respect to \(\tr\) and the equality
\begin{equation}\label{l4.6:e1}
d_{\tr}(\tld{x}, \tld{t}) = 0
\end{equation}
holds, then we have
\begin{equation}\label{l4.6:e2}
\tld{t} \in Seq(X, \tr).
\end{equation}
\end{lemma}

\begin{proof}
Suppose that \(\tx\) and \(\tld{t}\) are mutually stable with respect to \(\tr\) and \eqref{l4.6:e1} holds. Let \(\sstable{X}{r}\) be a maximal self-stable subset of \(Seq(X, \tr)\) such that \(\tx \in \sstable{X}{r}\). Then, by Lemma~\ref{l1.2.5}, \(\tld{t}\) and \(\tld{y}\) are mutually stable for every \(\tld{y} \in \sstable{X}{r}\). By statement~\ref{p1.2.2s5} of Proposition~\ref{p1.2.2}, we have the inclusion \(\sstable{X}{r}^{0} \subseteq \sstable{X}{r}\). Hence, \(\tld{t}\) and \(\tld{z}\) are mutually stable with respect to \(\tr\) for every \(\tld{z} \in \sstable{X}{r}^{0}\). Now \eqref{l4.6:e2} follows from statement~\ref{p1.2.2s6} of Proposition~\ref{p1.2.2}.
\end{proof}

\begin{remark}\label{r3.6}
The set $\nu_{0} = \nu_{0, \tr} := \sstable{X}{r}^{0}$ is a common distinguished point of all pretangent spaces $\pretan{X}{r}$ (with given scaling sequence $\tr$). This point can be informally described as follows: ``The points of pretangent space $\pretan{X}{r}$ are infinitely removed from the points of initial space $X$, but $\pretan{X}{r}$ contains a unique point $\nu_{0, \tr}$ which is close to $X$ as much as possible.'' Example~2.2 from \cite{BD2018} shows that a pretangent space \(\pretan{X}{r}\) can be one-point, \(\pretan{X}{r} = \{\nu_0\}\), if \((X, d)\) is sparse at infinity.
\end{remark}

In the next proposition, we define the mappings \(\varphi_{\tr'} \colon \sstable{X}{r} \to \tilde{X}_{\infty, \tr'}\) and \(em' \colon \pretan{X}{r} \to \pretan{X}{r'}\) as in \eqref{e1.1.8}.

\begin{proposition}\label{p1.3.1}
Let $(X, d)$ be an unbounded metric space, $\tr$ be a scaling sequence, $\tilde{\mathbf{R}}$ be the set of all subsequences of $\tr$ and let \(\sstable{X}{r}\) be a maximal self-stable subset of \(Seq(X, \tr)\). Let us consider a pretangent space $\pretan{X}{r}$ corresponding to $\sstable{X}{r}$. Write
\[
\tilde{X}'_{\infty, \tr} = \{(x_{n_k})_{k\in \NN} \colon \seq{x_n}{n} \in \sstable{X}{r}\}
\]
for every \(\tr' = (r_{n_k})_{k\in \NN} \in \tilde{\mathbf{R}}\). Then the following statements are equivalent:
\begin{enumerate}
\item\label{p1.3.1s1} $\pretan{X}{r}$ is tangent;
\item\label{p1.3.1s2} \(\tilde{X}'_{\infty, \tr}\) is maximal self-stable subset of \(Seq(X, \tr')\) for every \(\tr' \in \tilde{\mathbf{R}}\);
\item\label{p1.3.1s3} For every $\tr' \in \tilde{\mathbf{R}}$ and every $\tilde{X}_{\infty, \tr'}$ satisfying \(\tilde{X}_{\infty, \tr'} \supseteq \tilde{X}'_{\infty, \tr}\), the mapping $\varphi_{\tr'} \colon \sstable{X}{r} \to \tilde{X}_{\infty, \tr'}$ is onto;
\item\label{p1.3.1s4} For every $\tr' \in \tilde{\mathbf{R}}$ there is $\tilde{X}_{\infty, \tr'}$ such that \(\tilde{X}_{\infty, \tr'} \supseteq \tilde{X}'_{\infty, \tr}\) and the mapping $\varphi_{\tr'} \colon \sstable{X}{r} \to \tilde{X}_{\infty, \tr'}$ is onto.
\end{enumerate}
\end{proposition}
\begin{proof}
$\ref{p1.3.1s1} \Rightarrow \ref{p1.3.1s2}$. Suppose $\pretan{X}{r}$ is tangent. Let $(n_k)_{k\in \NN} \subset \NN$ be strictly increasing and let $\tr' := (r_{n_k})_{k\in\NN}$. Write \(\sstable{X}{r'}\) for the maximal self-stable subset of \(Seq(X, \tr')\) satisfying \(\tilde{X}'_{\infty, \tr} \subseteq \tilde{X}_{\infty, \tr'}\). It suffices to show that
$$
\tilde{X}'_{\infty, \tr} \supseteq \tilde{X}_{\infty, \tr'}.
$$
This inclusion holds if and only if for every $\ty = \seq{y_k}{k} \in \tilde{X}_{\infty, \tr'}$ there exists $\tilde{t} = \seq{t_n}{n} \in \sstable{X}{r}$ such that
\begin{equation}\label{p1.3.1e1}
y_k = t_{n_k}
\end{equation}
for every $k\in\NN$. Since \(\pretan{X}{r}\) is tangent, \(em'\colon \pretan{X}{r} \rightarrow \pretan{X}{r'}\) is an isometry. Consequently, for every $\ty=\seq{y_k}{k}\in\tilde{X}_{\infty, \tr'}$, there is $\tx = \seq{x_n}{n} \in \sstable{X}{r}$ such that $\pi'(\ty)=\pi(em'(\tx))$. This equality and the commutativity of diagram \eqref{e1.1.8} imply $\pi'(\varphi_{\tr'}(\tx)) = \pi'(\ty)$. By definition of the natural projection \(\pi'\), we can rewrite the last equality as
\begin{equation}\label{p1.3.1e2}
\lim_{k\to\infty}\frac{d(x_{n_k}, y_{k})}{r_{n_k}}=0.
\end{equation}
Let us define the sequence $\tilde t=\seq{t_n}{n}$ by the rule:
\begin{equation}\label{p1.3.1e3}
t_n:= \begin{cases}
y_n & \mbox{if } n\in\{n_1, n_2, n_3, \ldots\}\\
x_n & \mbox{otherwise}.
\end{cases}
\end{equation}
Then \eqref{p1.3.1e1} evidently holds. From \eqref{p1.3.1e2} and \eqref{p1.3.1e3} it follows that $\tilde t \in \tilde{X}_{\infty}$ and $\lim\limits_{n\to\infty}\frac{d(x_n, t_n)}{r_n} = 0$. The last equality, Lemma~\ref{l1.2.4} and Lemma~\ref{l1.2.5} imply $\tilde{t} \in \sstable{X}{r}$.

\(\ref{p1.3.1s2} \Rightarrow \ref{p1.3.1s3}\). Let \ref{p1.3.1s2} hold. Then, for every \(\tr' \in \tldbf{R}\), the inclusion \(\sstable{X}{r'} \supseteq \sstable{X}{r}'\) and the maximality of \(\sstable{X}{r'}\) and \(\sstable{X}{r}'\) imply the equality $\tilde{X}'_{\infty, \tr} = \tilde{X}_{\infty, \tr'}$ and, consequently, $\varphi_{\tr'}:\sstable{X}{r} \rightarrow \tilde{X}_{\infty, \tr'}$ is onto. 

\(\ref{p1.3.1s3} \Rightarrow \ref{p1.3.1s4}\). This implication is trivial.

\(\ref{p1.3.1s4} \Rightarrow \ref{p1.3.1s1}\). Let \ref{p1.3.1s4} hold, \(\tr' \in \tldbf{R}\), \(\sstable{X}{r'} \supseteq \sstable{X}{r}'\) and let \(\varphi_{\tr'} \colon \sstable{X}{r} \to \sstable{X}{r'}\) be a surjection. By \eqref{e1.1.8}, \(\varphi_{\tr'}\) is a distance preserving mapping of pseudometric spaces \((\sstable{X}{r}, d_{\tr})\) and \((\sstable{X}{r'}, d_{\tr'})\). It follows directly from Definition~\ref{d2.5} that every distance preserving surjection is pseudoisometry. Thus, \(\varphi_{\tr'}\) is a pseudoisometry. Since diagram \eqref{e1.1.8} is commutative and \(\varphi_{\tr'}\) is a pseudoisometry, the mapping \(em' \colon \pretan{X}{r} \to \pretan{X}{r'}\) is an isometry by Corollary~\ref{c2.13}. Thus, \(\pretan{X}{r}\) is asymptotically tangent to \((X, d)\).
\end{proof}

\begin{remark}\label{r4.9}
Proposition~\ref{p1.3.1} is a modification of the corresponding statements from~\cite{DAK2013, ADK2016A}.
\end{remark}

\section{Gluing pretangent spaces}
\label{sec5}

Let \((X, d)\) be an unbounded metric space and let \(\tr = \seq{r_n}{n}\) be a scaling sequence. Write
\begin{equation}\label{e2.10}
d^{\tr} (\tx, \ty) = \limsup_{n\to\infty} \frac{d(x_n, y_n)}{r_n}
\end{equation}
for every pair \(\tx\), \(\ty \in Seq(X, \tr)\). From~\eqref{e1.1.2} it follows that
\[
d^{\tr}(\tx, \ty) \leqslant \tld{d}_{\tr}(\tx) + \tld{d}_{\tr}(\ty).
\]
Hence, \(d^{\tr}(\tx, \ty)\) is finite for all \(\tx\), \(\ty \in Seq(X, \tr)\). It is also clear that \(d^{\tr}(\tx, \tx) = 0\) holds for every \(\tx \in Seq(X, \tr)\). Now using the inequality
\begin{equation}\label{e2.3}
\limsup_{n\to\infty}\frac{d(x_n, z_n)}{r_n} \le \limsup_{n\to\infty} \frac{d(x_n, y_n)}{r_n} + \limsup_{n\to\infty} \frac{d(y_n, z_n)}{r_n},
\end{equation}
we see that the mapping
\[
d^{\tr} \colon Seq^2(X, \tr) \to \RR
\]
is a pseudometric on \(Seq(X, \tr)\). 

\begin{conjecture}\label{con5.1}
The pseudometric space \((Seq(X, \tr), d^{\tr})\) is complete for every unbounded metric space and every scaling sequence \(\tr\). 
\end{conjecture}

\begin{proposition}\label{p5.3}
Let \((X, d)\) be an unbounded metric space, \((n_k)_{k \in \NN} \subset \NN\) be a strictly increasing sequence, \(\tr = (r_n)_{n \in \NN}\) be a scaling sequence and \(\tr' = (r_{n_k})_{k \in \NN}\) be the corresponding subsequence of \(\tr\). Then there is a unique mapping and \(\bm{\varphi}^{\tr'} \colon Seq(X, \tr) \to Seq(X, \tr')\) such that the following diagram
\begin{equation}\label{e2.13}
\ctdiagram{
\ctv 0,30: {\sstable{X}{r}} 
\ctv 120,30: {\sstable{X}{r'}}
\ctv 0,-30: {Seq(X, \tr)} 
\ctv 120,-30: {Seq(X, \tr')}
\ctet 0,30,120,30:{\varphi_{\tr'}}
\ctet 0,-30,120,-30:{\bm{\varphi}^{\tr'}}
\ctel 0,30, 0,-30:{In_{\tr}}
\cter 120,30, 120,-30:{In_{\tr'}}
}
\end{equation}
is commutative for every maximal self-stable \(\sstable{X}{r}\) and every maximal self-stable \(\sstable{X}{r'}\) satisfying \eqref{e1.1.7}, where \(In_{\tr}\) and \(In_{\tr'}\) are the corresponding inclusion maps,
\[
In_{\tr} (\tx) = \tx \text{ and } In_{\tr'} (\ty) = \ty
\]
for all \(\tx \in \sstable{X}{r}\), \(\ty \in \sstable{X}{r'}\), and \(\varphi_{\tr'}\) is defined as in \eqref{e1.1.8}. 
\end{proposition}

\begin{proof}
If we define \(\bm{\varphi}^{\tr'} \colon Seq(X, \tr) \to Seq(X, \tr')\) as \(\bm{\varphi}^{\tr'}(\tx) = \seq{x_{n_k}}{k}\) for every \(\tx = \seq{x_n}{n} \in Seq(X, \tr)\), then diagram~\eqref{e2.13} is evidently commutative for all \(\sstable{X}{r}\) and \(\sstable{X}{r'}\) satisfying \eqref{e1.1.7}. Since the set of all \(\sstable{X}{r}\) is a cover of \(Seq(X, \tr)\), the desirable mapping \(\bm{\varphi}^{\tr'}\) is unique.
\end{proof}

Let us denote by \((\mathbf{\Omega}_{\infty, \tr}^X, \bm{\rho})\) the metric identification of the pseudometric space \((Seq(X, \tr), d^{\tr})\). It is easy to prove that the equivalence
\[
(d^{\tr}(\tx, \ty) = 0) \Leftrightarrow \left(\parbox{.6\textwidth}{\(\tx\) and \(\ty\) are mutually stable with respect to \(\tr\) and \(d_{\tr}(\tx, \ty) = 0\)}\right)
\]
is valid for all \(\tx\), \(\ty \in Seq(X, \tr)\). Hence, from Lemma~\ref{l1.2.5} it follows that each asymptotically pretangent \(\pretan{X}{r}\) is a subset of \(\mathbf{\Omega}_{\infty, \tr}^X\),
\begin{equation}\label{e2.11}
\pretan{X}{r} \subseteq \mathbf{\Omega}_{\infty, \tr}^X.
\end{equation}
Moreover, if \(\pretan{X}{r}\) is the metric identification of some \(\sstable{X}{r}\), then, for all \(\tx\), \(\ty \in \sstable{X}{r}\), we have
\[
\dist{x}{y} = \lim_{n\to\infty} \frac{d(x_n, y_n)}{r_n} = \limsup_{n\to\infty} \frac{d(x_n, y_n)}{r_n} = d^{\tr}(\tx, \ty)
\]
because the set \(\sstable{X}{r}\) is a self-stable subset of \(Seq(X, \tr)\). Consequently, for every asymptotically pretangent space \((\pretan{X}{r}, \rho)\), the metric \(\rho\) is a restriction of the metric \(\bm{\rho}\) on the set \(\pretan{X}{r}\),
\[
\rho = \bm{\rho}|_{\pretan{X}{r} \times \pretan{X}{r}}.
\]
Since, for every \(\tx \in Seq(X, \tr)\), the one-point set \(\{\tx\}\) is a self-stable subset of \(Seq(X, \tr)\) and, consequently, belongs to some \(\sstable{X}{r}\), the set of all \(\pretan{X}{r}\) is a cover of the metric space \((\bfpretan{X}{r}, \bm{\rho})\).

Moving on to the metric identifications of the pseudometric spaces \((Seq(X, \tr), d^{\tr})\) and \((Seq(X, \tr'), d^{\tr'})\), we obtain the unique mapping \(\mathbf{Em}' \colon \bfpretan{X}{r} \to \bfpretan{X}{r'}\) with the commutative diagram 
\begin{equation}\label{e2.14}
\ctdiagram{
\ctv 0,30: {Seq(X, \tr)} 
\ctv 120,30: {\mathbf{\Omega}_{\infty, \tr}^X}
\ctv 0,-30: {Seq(X, \tr')} 
\ctv 120,-30: {\mathbf{\Omega}_{\infty, \tr'}^X}
\ctet 0,30,120,30:{\bm{\pi}}
\ctet 0,-30,120,-30:{\bm{\pi}'}
\ctel 0,30, 0,-30:{\bm{\varphi}^{\tr'}}
\cter 120,30, 120,-30:{\mathbf{Em}'}
},
\end{equation}
where \(\bm{\varphi}^{\tr'}\) is defined as in \eqref{e2.13}, and
\begin{align*}
\bm{\pi} (\tx) &:= \{\ty \in Seq(X, \tr) \colon d^{\tr} (\tx, \ty) = 0\},\\
\bm{\pi}' (\tz) &:= \{\tld{t} \in Seq(X, \tr') \colon d^{\tr'} (\tld{t}, \tz) = 0\}
\end{align*}
for every \(\tx \in Seq(X, \tr)\) and every \(\tz \in Seq(X, \tr')\) (cf.~\eqref{e1.1.9}).

Similarly to diagram \eqref{e2.13}, we obtain the commutative diagrams
\begin{equation}\label{e2.15}
\ctdiagram{
\ctv 0,30: {\sstable{X}{r}} 
\ctv 120,30: {\pretan{X}{r}}
\ctv 0,-30: {Seq(X, \tr)} 
\ctv 120,-30: {\mathbf{\Omega}_{\infty, \tr}^X}
\ctet 0,30,120,30:{\pi}
\ctet 0,-30,120,-30:{\bm{\pi}}
\ctel 0,30, 0,-30:{In_{\tr}}
\cter 120,30, 120,-30:{Em}
}
\end{equation}
and
\begin{equation}\label{e2.16}
\ctdiagram{
\ctv 0,30: {\sstable{X}{r'}} 
\ctv 120,30: {\pretan{X}{r'}}
\ctv 0,-30: {Seq(X, \tr')} 
\ctv 120,-30: {\mathbf{\Omega}_{\infty, \tr'}^X}
\ctet 0,30,120,30:{\pi'}
\ctet 0,-30,120,-30:{\bm{\pi}'}
\ctel 0,30, 0,-30:{In_{\tr'}}
\cter 120,30, 120,-30:{Em'}
}
\end{equation}
for the identical embeddings 
\[
Em \colon \pretan{X}{r} \to \mathbf{\Omega}_{\infty, \tr}^X \quad \text{and} \quad Em' \colon \pretan{X}{r'} \to \mathbf{\Omega}_{\infty, \tr'}^X.
\]

Since the set of all pretangent spaces \(\pretan{X}{r}\) is a cover of the metric space \((\mathbf{\Omega}_{\infty, \tr}^X, \bm{\rho})\), the mapping \(\mathbf{Em}' \colon \bfpretan{X}{r} \to \bfpretan{X}{r'}\) (see \eqref{e2.14}) can also be characterized as a unique mapping for which the diagram
\begin{equation}\label{e2.17}
\ctdiagram{
\ctv 0,30: {\pretan{X}{r}} 
\ctv 120,30: {\pretan{X}{r'}}
\ctv 0,-30: {\mathbf{\Omega}_{\infty, \tr}^X}
\ctv 120,-30: {\mathbf{\Omega}_{\infty, \tr'}^X}
\ctet 0,30,120,30:{em'}
\ctet 0,-30,120,-30:{\mathbf{Em}'}
\ctel 0,30, 0,-30:{Em}
\cter 120,30, 120,-30:{Em'}
}
\end{equation}
is commutative for all pretangent spaces \(\pretan{X}{r}\), where \(em' \colon \pretan{X}{r} \to \pretan{X}{r'}\) is defined as in \eqref{e1.1.8}. 

It is interesting to note that in diagram \eqref{e2.17} the mappings \(Em\), \(Em'\) and \(em'\) are isometric embeddings, but \(\mathbf{Em}'\) is a Lipschitz mapping with constant \(1\), i.e., the inequality 
\[
\bm{\rho}'(\mathbf{Em}'(\alpha), \mathbf{Em}'(\beta)) \leq \bm{\rho}(\alpha, \beta)
\]
holds for all \(\alpha\), \(\beta \in (\mathbf{\Omega}_{\infty, \tr'}^X, \bm{\rho})\). Moreover, unlike Proposition~\ref{p1.3.1}, \(\mathbf{Em}' \colon \mathbf{\Omega}_{\infty, \tr}^X \to \mathbf{\Omega}_{\infty, \tr'}^X\) may be neither injective nor surjective even if \(\bm{\varphi}^{\tr'} \colon Seq(X, \tr) \to Seq(X, \tr')\) (see \eqref{e2.14}) is onto for every \(\tr' \in \tilde{\mathbf{R}}\).

For more detailed description of situation, we shall use the concept of porosity at infinity.

\begin{definition}\label{d5.1}
Let \(E \subseteq \RR_{+}\). The porosity of \(E\) at infinity is the quantity
\begin{equation}\label{d5.1:e1}
p^{+}(E, \infty) = \limsup_{h \to \infty} \frac{l(\infty, h, E)}{h},
\end{equation}
where \(l(\infty, h, E)\) is the length of the longest interval in the set \([0, h] \setminus E\). The set \(E\) is \emph{porous at infinity} if \(p^{+}(E, \infty) > 0\), and, respectively, \(E\) is \emph{nonporous at infinity} if \(p^{+}(E, \infty) = 0\).
\end{definition}

Let \((X, d)\) be a metric space, and let \(p \in X\). Write 
\begin{equation}\label{e5.9}
Sp(X) = \{d(x, p) \colon x \in X\}.
\end{equation}

\begin{theorem}\label{t5.2}
Let \((X, d)\) be an unbounded metric space and let \(p \in X\). Then the following statements are equivalent:
\begin{enumerate}
\item \label{t5.2:s1} The set \(Sp(X)\) is nonporous at infinity, \(p^{+}(Sp(X), \infty) = 0\).
\item \label{t5.2:s2} The mapping \(\bm{\varphi}^{\tr'} \colon Seq(X, \tr) \to Seq(X, \tr')\) is a surjection for every scaling sequence \(\tr\) and every \(\tr' \in \tldbf{R}(\tr)\).
\item \label{t5.2:s3} The equality
\begin{equation}\label{t5.2:e0}
Sp(\bfpretan{X}{r_1}) = Sp(\bfpretan{X}{r_2})
\end{equation}
holds for any two scaling sequences \(\tr_1\) and \(\tr_2\), where 
\[
Sp(\bfpretan{X}{r_i}) := \{\bm{\rho}_i(\nu, \nu_{0, \tr_i}) \colon \nu \in \bfpretan{X}{r_i}\},
\]
\((\bfpretan{X}{r_i}, \bm{\rho}_i)\) is the metric identification of the pseudometric space \((Seq(X, \tr_i), d^{\tr_i})\) and 
\[
\nu_{0, \tr_i} = \sstable{X}{r_i}^{0} = \{\tx \in Seq(X, \tr_i) \colon \Dist[r_i]{x} = 0\}
\]
is the distinguished point of \(\bfpretan{X}{r_i}\), \(i = 1\), \(2\).
\end{enumerate}
\end{theorem}

\begin{proof}
\(\ref{t5.2:s1} \Rightarrow \ref{t5.2:s2}\). Suppose that \(Sp(X)\) is nonporous at infinity, \(p^{+}(Sp(X), \infty) = 0\). Let \(\tr = \seq{r_n}{n}\) be an arbitrary scaling sequence, \(\tr' = \seq{r_{n_k}}{k} \in \tldbf{R}(\tr)\) and let \(\tx = \seq{x_k}{k} \in Seq(X, \tr')\). We must find \(\ty = \seq{y_n}{n} \in Seq(X, \tr)\) such that \(\bm{\varphi}^{\tr'}(\ty) = \tx\), i.e., 
\begin{equation}\label{t5.2:e1}
\seq{y_{n_k}}{k} = \seq{x_k}{k}.
\end{equation}

If the sequence \(\tx = \seq{x_k}{k}\) belongs to the set 
\[
\sstable{X}{r'}^{0} = \{\tx \in Seq(X, \tr') \colon \tld{d}_{\tr'}(\tx) = 0\},
\]
then the existence of a desired \(\ty\) follows from statement \ref{p1.2.2s9} of Proposition~\ref{p1.2.2}. 

Let us consider the case \(\tld{d}_{\tr'}(\tx) > 0\). To simplify notation, denote
\begin{equation}\label{t5.2:e2}
t := \tld{d}_{\tr'}(\tx).
\end{equation}
Using \eqref{d5.1:e1}, we can rewrite the equality \(p^{+}(Sp(X), \infty) = 0\) in the form
\[
\lim_{h \to \infty} \frac{l(\infty, h, Sp(X))}{h} = 0.
\]
Since \(t > 0\), the last equality implies
\begin{equation}\label{t5.2:e3}
0 = \lim_{n \to \infty} \frac{l(\infty, tr_n, Sp(X))}{tr_n} = \lim_{n \to \infty} \frac{l(\infty, tr_n, Sp(X))}{r_n}.
\end{equation}
Let us consider a sequence \(\seq{s_m}{m} \subseteq (0, 1)\) such that \(\lim_{m \to \infty} s_m = 1\) and \(s_m < s_{m+1}\) for every \(m \in \NN\). It follows from the definition of \(l(\infty, h, E)\), \eqref{t5.2:e3} and the definition of \(\seq{s_m}{m}\) that there is a sequence 
\[
j_1 < j_2 < \ldots < j_k < \ldots
\]
of positive integer numbers such that 
\[
(s_1 t r_n, t r_n) \cap Sp(X) \neq \varnothing
\]
for all \(n > j_1\),
\[
(s_2 t r_n, t r_n) \cap Sp(X) \neq \varnothing
\]
for all \(n > j_2\), \ldots, 
\[
(s_k t r_n, t r_n) \cap Sp(X) \neq \varnothing
\]
for all \(n > j_k\) nd so on. Hence, we can find a sequence \(\tz = \seq{z_n}{n} \subseteq X\) satisfying the double inequality
\begin{equation}\label{t5.2:e4}
s_k t \leqslant \frac{d(p, z_n)}{r_n} \leqslant t 
\end{equation}
whenever \(n > j_k\). From \eqref{t5.2:e4} it follows that
\begin{equation}\label{t5.2:e5}
s_k t \leqslant \liminf_{n \to \infty} \frac{d(p, z_n)}{r_n} \leqslant \limsup_{n \to \infty} \frac{d(p, z_n)}{r_n} \leqslant t.
\end{equation}
Now letting \(k \to \infty\) and using the limit relation \(\lim_{k \to \infty} s_k = 1\), we obtain from \eqref{t5.2:e5} the equality
\begin{equation}\label{t5.2:e6}
t = \lim_{n \to \infty} \frac{d(p, z_n)}{r_n}.
\end{equation}

Let us define a sequence \(\ty = \seq{y_n}{n}\) by the rule
\begin{equation}\label{t5.2:e7}
y_n = \begin{cases}
x_n & \text{if } n \in \{n_1, \ldots, n_k, \ldots\},\\
z_n & \text{otherwise}.
\end{cases}
\end{equation}
Limit relations \eqref{t5.2:e2} and \eqref{t5.2:e6} imply 
\[
\lim_{n \to \infty} \frac{d(y_n, p)}{r_n} = t.
\]
Hence, \(\ty\) belongs to \(Seq(X, \tr)\). Now \eqref{t5.2:e1} follows from \eqref{t5.2:e7}.

\(\ref{t5.2:s2} \Rightarrow \ref{t5.2:s3}\). Suppose statement \ref{t5.2:s2} holds. Let \(\tr_1 = \seq{r_{1, n}}{n}\) and \(\tr_2 = \seq{r_{2, n}}{n}\) be two arbitrary scaling sequences and let \(\tr = \seq{r_m}{m}\) be a sequence defined by the rule
\[
r_m = \begin{cases}
r_{1, n} & \text{if \(m\) is odd and \(m = 2n-1\)},\\
r_{2, n} & \text{if \(m\) is even and \(m = 2n\)},\\
\end{cases}
\]
i.e., 
\[
(r_1, r_2, r_3, r_4, r_5, \ldots) = (r_{1, 1}, r_{2, 1}, r_{1, 2}, r_{2, 2}, r_{1, 3}, \ldots).
\]
Then \(\tr_1\) and \(\tr_2\) are subsequences of \(\tr\). Moreover, the equalities
\[
\lim_{n \to \infty} r_{1, n} = \lim_{n \to \infty} r_{2, n} = \infty
\]
imply the equality \(\lim_{m \to \infty} r_{m} = \infty\). Thus, \(\tr\) also is a scaling sequence. 

To prove equality \eqref{t5.2:e0}, let us consider an arbitrary \(\nu_1 \in \bfpretan{X}{r_1}\). Then there is \(\tx^1 \in Seq(X, \tr_1)\) such that
\[
\bm{\pi}^1(\tx^{1}) = \nu_1,
\]
where \(\bm{\pi}^1 \colon Seq(X, \tr_1) \to \bfpretan{X}{r_1}\) is the natural projection of the pseudometric space \((Seq(X, \tr_1), d^{\tr_1})\) on its metric identification \((\bfpretan{X}{r_1}, \bm{\rho}_1)\). Using statement \ref{t5.2:s2} with \(\tr' = \tr_1\), we can find \(\tx \in Seq(X, \tr)\) such that
\begin{equation*}
\bm{\varphi}^{\tr_1}(\tx) = \tx^{1}.
\end{equation*}
Let \(\tz\) be an arbitrary point of \(\sstable{X}{r_1}^0\). Then the equality
\[
\Dist{x} = \Dist[r_1]{x^1}
\]
holds by \eqref{e3.3} and, in addition, we have
\[
\Dist[r_1]{x^1} = d_{\tr_1}(\tx^1, \tz)
\]
by statement \ref{p1.2.2s9} from Proposition~\ref{p1.2.2}. Since \(\tx^1\) and \(\tz\) are mutually stable with respect to \(\tr_1\), we also have
\[
d_{\tr_1}(\tx^1, \tz) = d^{\tr_1}(\tx^1, \tz).
\]
Example~\ref{ex2.5} with \((Y, \rho) = (Seq(X, \tr_1), d^{\tr_1})\) implies
\begin{equation}\label{t5.2:e14}
\Dist{x} = d_{\tr_1}(\tx^1, \tz) = \bm{\rho}_1(\nu_1, \nu_{0, \tr_1}).
\end{equation}

Let us define now \(\tx^2 \in Seq(X, \tr_2)\) as 
\[
\tx^2 = \bm{\varphi}^{\tr_2} (\tx),
\]
where \(\tx\) is the same as in \eqref{t5.2:e14}, and write
\[
\nu_2 := \bm{\pi}^2(\tx^2).
\]
Then, similarly to \eqref{t5.2:e14}, we obtain
\[
\Dist{x} = \bm{\rho}_2(\nu_2, \nu_{0, \tr_1}).
\]
Thus, for every \(\nu_1 \in \bfpretan{X}{r_1}\), we have \(\bm{\rho}_1(\nu_1, \nu_{0, \tr_1}) \in Sp(\bfpretan{X}{r_2})\), that implies
\[
Sp(\bfpretan{X}{r_1}) \subseteq Sp(\bfpretan{X}{r_2}).
\]
The converse inclusion 
\[
Sp(\bfpretan{X}{r_2}) \subseteq Sp(\bfpretan{X}{r_1})
\]
can be proved similarly. Equality~\eqref{t5.2:e0} follows.

\(\ref{t5.2:s3} \Rightarrow \ref{t5.2:s1}\). Let statement \ref{t5.2:s3} hold. We need prove the equality \(p^{+}(Sp(X), \infty) = 0\). Suppose contrary that 
\[
t:= p^{+}(Sp(X), \infty) > 0.
\]
Then there exists a sequence of open intervals 
\[
(a_1, b_1), \ldots, (a_n, b_n), \ldots 
\]
such that:
\begin{gather}
\label{t5.2:e8.1}
\lim_{n \to \infty} b_n = \infty; \\
0 < a_1 < b_1 < a_2 < b_2 < \ldots < a_n < b_n < \ldots; \notag\\
\label{t5.2:e8}
(a_n, b_n) \cap Sp(X) = \varnothing
\end{gather}
for every \(n \in \NN\); 

\noindent and
\begin{equation}\label{t5.2:e9}
\lim_{n \to \infty} \frac{b_n - a_n}{b_n} = p^{+}(Sp(X), \infty) = t > 0.
\end{equation}
From~\eqref{t5.2:e8.1} it follows that we may define a scaling sequence \(\tr_1 = \seq{r_{1, n}}{n}\) by \(r_{1, n} = b_n\), \(n \in \NN\). Let \(\tz = \seq{z_n}{n}\) belong to \(Seq(X, \tr_1)\). Then, using \eqref{t5.2:e8}, we see that there is a subsequence \(\seq{z_{n_k}}{k}\) of \(\tz\) such that either
\begin{equation}\label{t5.2:e10}
d(p, z_{n_k}) \leqslant a_{n_k}
\end{equation}
for all \(k \in \NN\), or 
\begin{equation}\label{t5.2:e11}
d(p, z_{n_k}) \geqslant b_{n_k}
\end{equation}
for all \(k \in \NN\). In the first case, using \eqref{t5.2:e9} and \eqref{t5.2:e10}, we obtain
\[
\Dist[r_1]{z} = \lim_{n \to \infty} \frac{d(z_n, p)}{b_n} = \lim_{k \to \infty} \frac{d(z_{n_k}, p)}{b_{n_k}} \leqslant \lim_{n \to \infty} \frac{a_n}{b_n} = 1 - t,
\]
because
\[
t = \lim_{n \to \infty} \frac{b_n - a_n}{b_n} = 1 - \lim_{n \to \infty} \frac{a_n}{b_n}.
\]
Analogously, if we have \eqref{t5.2:e11} for all \(k \in \NN\), then the inequality
\[
\Dist[r_1]{z} \geqslant 1
\]
holds. Thus, for every \(\tz \in Seq(X, \tr_1)\), we have
\[
\Dist[r_1]{z} \notin (1-t, 1),
\]
that implies 
\begin{equation}\label{t5.2:e12}
Sp(\bfpretan{X}{r_1}) \cap (1-t, 1) = \varnothing.
\end{equation}

Let \(c\) be a point of the open interval \((1-t, 1)\). By statement~\ref{p1.2.1s2} of Proposition~\ref{p1.2.1}, there is a scaling sequence \(\tr = \seq{r_n}{n}\) and \(\tx \in Seq(X, \tr)\) such that \(\Dist{x} = 1\). Then, for the scaling sequence \(\tr_2 = \seq{r_{2, n}}{n}\) defined by 
\[
r_{2, n} = \frac{1}{c} r_n, \quad n \in \NN,
\]
we have the equality 
\[
\Dist[r_2]{x} = c.
\]
Thus, \(c\) belongs to \(Sp(\bfpretan{X}{r_2}) \cap (1-t, 1)\). The last statement and \eqref{t5.2:e12} imply
\[
Sp(\bfpretan{X}{r_1}) \neq Sp(\bfpretan{X}{r_2})
\]
contrary to statement \ref{t5.2:s3}.
\end{proof}

\begin{remark}\label{r5.3}
It should be noted here that the porosity of the set \(Sp(X)\) at infinity does not depend on the choice of point \(p\) in \eqref{e5.9}, although we may have 
\[
\{d(x, a) \colon x \in X\} \neq \{d(x, b) \colon x \in X\}
\]
for different \(a\), \(b \in X\).
\end{remark}

The standard definition of porosity at a finite point can be found in \cite{Thomson1985}. See \cite{ADK2016A, BD2014RAE, BDK2013JA, BD2014AASFM} for some applications of the porosity to studies of the infinitesimal properties of metric spaces.

\section{Asymptotically pretangent spaces to real line and half-line}
\label{sec6}

In the first proposition we will describe the structure of \(\sstable{X}{r}\) for \(X = \RR\) or \(X = \RR_{+}\).

\begin{proposition}\label{p1.3.2}
Let $(X,d)$ be a metric space with $X = \RR$ or \(X = \RR_{+}\) endowed with the usual metric $d(x,y) = |x-y|$. Then the following statements hold for every scaling sequence \(\tr\), every maximal self-stable $\sstable{X}{r}$ and each $\tb = \seq{b_n}{n} \in \sstable{X}{r}$ satisfying
\begin{equation}\label{p1.3.2:e1}
\Dist{b} = \lim_{n\to\infty} \frac{|b_n|}{r_n} > 0.
\end{equation}
\begin{enumerate}
\item\label{p1.3.2:s1} For every $\ty = \seq{y_n}{n} \in \sstable{X}{r}$ there is a finite limit $\lim_{n\to\infty} \frac{y_n}{b_n}$ and, conversely, if, for $\ty \in \tilde{X}_{\infty}$, this limit exists and is finite, then $\ty \in \sstable{X}{r}$.
\item\label{p1.3.2:s2} For every two $\tx = \seq{x_n}{n}$ and $\ty = \seq{y_n}{n}$ from $\sstable{X}{r}$ the equality $d_{\tr} (\tx, \ty) = 0$ holds if and only if 
$$
\lim_{n\to\infty} \frac{x_n}{b_n} = \lim_{n\to\infty} \frac{y_n}{b_n}.
$$
\item\label{p1.3.2:s3} The pretangent space $\pretan{X}{r}$ corresponding to $\sstable{X}{r}$ is isometric to $(X, d)$ and tangent.
\end{enumerate}
\end{proposition}

\begin{proof}
Statements \ref{p1.3.2:s1}--\ref{p1.3.2:s3} were proved in Proposition~2.1 \cite{BD2019a} for the case when \(X = \RR_{+}\). Let us consider the case \(X = \RR\). 

Let \(\tr\) be a scaling sequence, \(\sstable{X}{r}\) be a maximal self-stable subset of \(Seq(X, \tr)\) and let \(\tb \in \sstable{X}{r}\) satisfy~\eqref{p1.3.2:e1}.

\(\ref{p1.3.2:s1}\). If $\ty = \seq{y_n}{n} \in \sstable{X}{r}$, then there are the finite limits
$$
\Dist{y} = \lim_{n\to\infty} \frac{|y_n|}{r_n} \text{ and } d_{\tr}(\tb, \ty) = \lim_{n\to\infty} \frac{|y_n - b_n|}{r_n}.
$$
For the case $\Dist{y} = 0$ we obtain 
$$
0 = \frac{\Dist{y}}{\Dist{b}} = \lim_{n\to\infty} \frac{|y_n|}{|b_n|} = \lim_{n\to\infty} \frac{y_n}{b_n}
$$
because $\Dist{b} \neq 0$. Suppose $\Dist[r]{y} \neq 0$, then
\begin{equation}\label{p1.3.2e1}
0 < \frac{\Dist{y}}{\Dist{b}} = \lim_{n\to\infty} \frac{|y_n|}{|b_n|} < \infty.
\end{equation}
Write $t = |t|\sign{t}$ for every $t \in \RR$, where
$$
\sign{t} = \begin{cases}
1 & \text{if } t > 0\\
0 & \text{if } t = 0\\
-1 & \text{if } t < 0.
\end{cases}
$$
It follows from~\eqref{p1.3.2e1}, that the limit $\lim_{n\to\infty} \frac{y_n}{b_n}$ exists if and only if there is $\lim_{n\to\infty} \frac{\sign{y_n}}{\sign{b_n}}$. If the last limit does not exist, then there are two infinite sequences $\tilde{n} = \seq{n_k}{k}$ and $\tilde{m} = \seq{m_k}{k}$ of natural numbers such that 
$$
\sign{y_{n_k}} = \sign{b_{n_k}} \quad \text{and} \quad \sign{y_{m_k}} = - \sign{b_{m_k}}
$$
for all $k \in \NN$. Consequently, we obtain
$$
\dist{y}{b} = \lim_{k\to\infty} \frac{|y_{n_k} - b_{n_k}|}{r_{n_k}} = \lim_{k\to\infty} \frac{\bigl|\left|y_{n_k}\right| - \left|b_{n_k}\right|\bigr|}{r_{n_k}} = \bigl|\Dist{y} - \Dist{b}\bigr|
$$
and, similarly, 
$$
\dist{y}{b} = \lim_{k\to\infty} \frac{|y_{m_k} - b_{m_k}|}{r_{m_k}} = \Dist{y} + \Dist{b}.
$$
Thus we have the equality
$$
\Dist{y} + \Dist{b} = \bigl|\Dist{y} - \Dist{b}\bigr|
$$
which implies
$$
\min\{\Dist{y}, \Dist{b}\} = 0,
$$
contrary to \eqref{p1.3.2e1}. It is shown that for every $\ty \in \sstable{X}{r}$ there is a finite limit $\lim_{n\to\infty} \frac{y_n}{b_n}$. 

Conversely, let $\ty \in \tilde{X}_{\infty}$ and 
\begin{equation}\label{p1.3.2e2}
\lim_{n\to\infty} \frac{y_n}{b_n} = c
\end{equation}
for some $c \in \RR$. The last equality and the inequality $\Dist{b} > 0$ imply that there is a finite limit
$$
\Dist{y} = \lim_{n\to\infty} \frac{|y_n|}{r_n}.
$$
Thus we must show that for every $\tx \in \sstable{X}{r}$ there is $\lim_{n\to\infty} \frac{|y_n - x_n|}{r_n}$, i.e., $\tx$ and $\ty$ are mutually stable w.r.t. $\tr$. It is easy to do. Since $\tx \in \sstable{X}{r}$, we have a finite limit 
\begin{equation}\label{p1.3.2e3}
\lim_{n\to\infty} \frac{x_n}{b_n} := k.
\end{equation}
Hence
\begin{equation}\label{p1.3.2e4}
\lim_{n\to\infty} \frac{\left|y_n - x_n\right|}{|r_n|} = \lim_{n\to\infty} \frac{|b_n|}{r_n} \left|\frac{x_n}{b_n} - \frac{y_n}{b_n}\right| = \Dist{b} |c-k|,
\end{equation}
where the constants $c$, $k$ are defined by~\eqref{p1.3.2e2} and~\eqref{p1.3.2e3}, respectively.

\(\ref{p1.3.2:s2}\). Statement $(ii)$ follows from~\eqref{p1.3.2e4}.

\(\ref{p1.3.2:s3}\). Statement $(i)$ implies that the sequence $\tr^* = \seq{r_n^*}{n}$ with
$$
r_n^* = r_n \sign{b_n}, \quad n \in \NN,
$$
belongs to $\sstable{X}{r}$. Considering $\tr^*$ instead of $\tb$ in \eqref{p1.3.2e2} and \eqref{p1.3.2e3}, we obtain the mapping $f \colon \sstable{X}{r} \to \RR$ with
$$
f(\tx) = \lim_{n\to\infty} \frac{x_n}{r_n^*}, \quad \tx = \seq{x_n}{n} \in \sstable{X}{r}.
$$
It is easy to see that there is a unique mapping $\psi \colon \pretan{X}{r} \to \RR$ such that the diagram
\begin{equation}\label{p1.3.2e5}
\ctdiagram{
\ctv 0,30: {\sstable{X}{r}} 
\ctv 120,30: {\pretan{X}{r}}
\ctv 120,-30: {\RR}
\ctet 0,30,120,30:{\pi}
\cter 120,30,120,-30:{\psi}
\cteb 0,30,120,-30:{f}
}
\end{equation}
is commutative, where $\pi$ is the natural projection (see \eqref{e1.1.8}). Since, $\Dist{r^*} = 1$, \eqref{p1.3.2e4} implies that $\psi$ is an isometry. It remains to prove that $\pretan{X}{r}$ is tangent. 

Let $\tilde{n} = \seq{n_k}{k}$ be a strictly increasing sequence of positive integer numbers and let $\tr' = \seq{r_{n_k}}{k}$ be the corresponding subsequence of the scaling sequence $\tr$. If $\sstable{X}{r'}$ is a maximal self-stable set such that
$$
\sstable{X}{r'} \supseteq \{\tx' \colon \tx \in \sstable{X}{r}\},
$$
then, by statement $(i)$, for every $\tx = \seq{x_k}{k} \in \sstable{X}{r'}$ there is a finite limit 
$$
\lim_{k \to \infty} \frac{x_k}{r_{n_k} \sign{b_{n_k}}}:= s.
$$
Define $\ty = \seq{y_n}{n} \in \tilde{X}_{\infty}$ by the rule:
$$
y_n:= \begin{cases}
x_{k} & \text{if there is $n_k$ such that $n_k=n$}\\
r_n \sign{b_n} & \text{otherwise}. 
\end{cases}
$$
It is clear that $\seq{y_{n_k}}{k} = \seq{x_{k}}{k}$. A simple calculation shows that 
$$
\lim_{n\to\infty} \frac{y_n}{r_n \sign{b_n}} = s.
$$
By statement $(i)$, $\ty$ belongs to $\sstable{X}{r}$. Using Proposition~\ref{p1.3.1} we see that $\pretan{X}{r}$ is tangent.
\end{proof}

Statement \ref{p1.3.2:s1} of Proposition~\ref{p1.3.2} implies the following.

\begin{corollary}\label{c4.3}
Let \((X, d)\) be the metric space with \(X = \RR_{+}\) endowed with the usual metric \(d(x, y) = |x - y|\). Then, for every scaling sequence \(\tr\), the set \(Seq(X, \tr)\) is self-stable and, consequently, \((X, d)\) has the unique pretangent space at infinity with respect to every \(\tr\).
\end{corollary}

The following lemma is a reformulation of Theorem~3.1 from \cite{BD2019}.

\begin{lemma}\label{l4.3}
Let \((X, d)\) be an unbounded metric space. Then all pretangent spaces \(\pretan{X}{r}\) are complete.
\end{lemma}

The next lemma is a modification of Lemma~2.2 from~\cite{DD2009}.

\begin{lemma}\label{l4.4}
Let \(Y_1\) and \(Y_2\) be isometric subspaces of the real line \(\RR\) endowed with the usual metric. Then \(\RR \setminus Y_1\) and \(\RR \setminus Y_2\) also are isometric as subspaces of \(\RR\).
\end{lemma}

\begin{proof}
The conclusion of the lemma is trivially valid if \(Y_1\) is an one-point set. Suppose that \(|Y_1| \geqslant 2\). 

Let \(f \colon Y_1 \to Y_2\) be an isometry of \(Y_1\) and \(Y_2\). It suffices to show that there is a self-isometry \(\Psi \colon \RR \to \RR\) such that \(f\) is the restriction \(\Psi|_A\) of \(\Psi\) on the set \(A\),
\begin{equation}\label{l4.4:e1}
f = \Psi|_A.
\end{equation}
Let \(a\) and \(b\) be different points of \(Y_1\). Using shifts and reflections of the real line \(\RR\), we can find a self-isometry \(\Psi_2 \colon \RR \to \RR\) such that
\begin{equation}\label{l4.4:e2}
a = \Psi_2(f(a)) \text{ and } b = \Psi_2(f(b)).
\end{equation}
Then the function \(Y_1 \xrightarrow{f} Y_2 \xrightarrow{\Psi_2|_{Y_2}} \RR\) is an isometric embedding of \(Y_1\) in \(\RR\), where \(\Psi_2|_{Y_2}\) is the restriction of \(\Psi_2\) on the set \(Y_2\). Let us denote the last embedding by \(\varphi\). If \(t\) is an arbitrary point of \(Y_1\), then, using \eqref{l4.4:e2}, we obtain
\begin{equation}\label{l4.4:e3}
|a - t| = |\varphi(a) - \varphi(t)| = |a - \varphi(t)|
\end{equation}
and
\begin{equation}\label{l4.4:e4}
|b - t| = |\varphi(b) - \varphi(t)| = |b - \varphi(t)|.
\end{equation}
Since the intersection of any two different ``spheres'' in \(\RR\) is either empty or contains exactly one point, \eqref{l4.4:e3} and \eqref{l4.4:e4} imply \(\varphi(t) = t\). Thus, \(\varphi\) is the restriction of the identical mapping \(\operatorname{Id} \colon \RR \to \RR\) on the set \(Y_1\). It implies equality \eqref{l4.4:e1} with \(\Psi = \Psi_2^{-1}\).
\end{proof}

The following theorem is one of the main results of the paper.

\begin{theorem}\label{t4.3}
Let \((X, d)\) be isometric to the real line \(\RR\) with the usual metric and let \(Y\) be unbounded subspace of \((X, d)\). Then the following conditions are equivalent:
\begin{enumerate}
\item \label{t4.3:c1} \(Y\) is isometric to \(\RR\) or \(\RR_{+}\);
\item \label{t4.3:c2} For every scaling sequence \(\tr = (r_n)_{n \in \NN}\), every pretangent space \(\pretan{Y}{r}\) is isometric to \(Y\).
\end{enumerate}
\end{theorem}

\begin{proof}
Statement~\ref{p1.3.2:s3} of Proposition~\ref{p1.3.2} implies the validity of the implication \(\ref{t4.3:c1} \Rightarrow \ref{t4.3:c2}\).

Let us prove the validity of \(\ref{t4.3:c2} \Rightarrow \ref{t4.3:c1}\). Suppose that condition \ref{t4.3:c2} holds. Let \(\tr = (r_n)_{n \in \NN}\) be a fixed scaling sequence, \(Y_{\infty, \tr}\) be a maximal self-stable subset of \(Sec(Y, \tr)\) and let \(\pretan{Y}{r}\) be the metric identification of \(Y_{\infty, \tr}\). For arbitrary \(k > 0\), we denote by \(k \tr\) the sequence \((kr_n)_{n \in \NN}\),
\begin{equation}\label{t4.3:e1}
k \tr := (kr_n)_{n \in \NN}.
\end{equation}
Considering \(k \tr\) as a new scaling sequence and using \eqref{t4.3:e1} it is easy to see that:
\begin{itemize}
\item The equality \(Seq(Y, \tr) = Seq(Y, k\tr)\) holds;
\item The equivalence
\begin{multline*}
(\text{\(\tx\) and \(\ty\) are mutually stable w.r.t \(\tr\)}) \\
\Leftrightarrow (\text{\(\tx\) and \(\ty\) are mutually stable w.r.t \(k\tr\)})
\end{multline*}
is valid for all \(\tx\), \(\ty \in Seq(Y, \tr)\).
\end{itemize}
Hence, \(Y_{\infty, \tr}\) is also a maximal subset of \(Seq(Y, k\tr)\). Therefore, we may re-designate \(Y_{\infty, \tr}\) as \(Y_{\infty, k\tr}\). It is also follows from \eqref{t4.3:e1} that
\begin{equation}\label{t4.3:e2}
(\tx \coherent{\tr} \ty) \Leftrightarrow (\tx \coherent{k\tr} \ty)
\end{equation}
is valid for all \(\tx\), \(\ty \in Seq(Y, \tr)\).

Now \(Y_{\infty, \tr} = Y_{\infty, k\tr}\) and \eqref{t4.3:e2} imply the equality \(\Omega_{\infty, \tr}^Y = \Omega_{\infty, k\tr}^Y\). It follows from \ref{t4.3:c2} that \((\pretan{Y}{r}, \rho)\) and \((\Omega_{\infty, k\tr}^Y, \rho_k)\) are isometric, where 
\[
\rho (\alpha, \beta) = d_{\tr} (\tx, \ty)
\]
for \(\tx \in \alpha \in \Omega_{\infty, \tr}^Y\), \(\ty \in \beta \in \Omega_{\infty, \tr}^Y\) and, analogously, 
\[
\rho_k (\gamma, \delta) = d_{k\tr} (\tld{u}, \tld{v})
\]
for \(\tld{u} \in \gamma \in \Omega_{\infty, k\tr}^Y\), \(\tld{v} \in \delta \in \Omega_{\infty, k\tr}^Y\) (see \eqref{e1.1.5}). Moreover, using \eqref{e1.1.1}, \eqref{e1.1.5} and \eqref{t4.3:e1}, we see that
\begin{equation}\label{t4.3:e3}
\frac{1}{k} \rho(\alpha, \beta) = \rho_k(\alpha, \beta)
\end{equation}
holds for all \(\alpha\), \(\beta \in \Omega_{\infty, \tr}^Y\).

By condition~\ref{t4.3:c2}, there exist isometric embeddings \(\Phi \colon \Omega_{\infty, \tr}^Y \to \RR\) and \(\Phi_k \colon \Omega_{\infty, k\tr}^Y \to \RR\) such that \(\Phi(\Omega_{\infty, \tr}^Y)\) and \(\Phi_k(\Omega_{\infty, k\tr}^Y)\) are isometric subspaces of \(\RR\). Write \(A := \{\Phi(\alpha) \colon \alpha \in \Omega_{\infty, \tr}^Y\}\). Then \eqref{t4.3:e3} implies
\begin{equation*}
\Phi_k (\Omega_{\infty, k\tr}^Y) = k^{-1} A = \{k^{-1} t \colon t \in A\}.
\end{equation*}
Consequently, the metric spaces \(\RR \setminus A\) and \(\RR \setminus k^{-1} A\) are also isometric for every \(k > 0\) by Lemma~\ref{l4.4}. 

Let us now prove that condition~\ref{t4.3:c1} is satisfied. Using the above introduced designations, we can rewrite this condition as: Either \(A\) is isometric to \(\RR\) or \(A\) is isometric to \(\RR_{+}\). 

First of all, we note that if \(\RR \setminus A\) is empty, then \(A = \RR\) holds. Let \(\RR \setminus A\) be nonempty.

Let us consider the case when \(\RR \setminus A\) is connected and unbounded. By Lemma~\ref{l4.3}, \(\RR \setminus A\) is open. Consequently, we have
\begin{equation}\label{t4.3:e5}
\RR \setminus A = (- \infty, \infty),
\end{equation}
or there is \(s \in \RR\) such that
\begin{equation}\label{t4.3:e6}
\RR \setminus A = (- \infty, s) \text{ or } \RR \setminus A = (s, \infty).
\end{equation}
Equality~\eqref{t4.3:e5} implies \(A = \varnothing\), contrary to the definition of pretangent metric spaces (see Remark~\ref{r3.6}). Hence, \eqref{t4.3:e6} holds. Passing in formula \eqref{t4.3:e6} from \(\RR \setminus A\) to \(A\), we obtain
\[
A = [s, \infty) \text{ or } A = (-\infty, s],
\]
that implies the isometricity of \(A\) and \(\RR_{+}\).

Thus, if \ref{t4.3:c1} does not hold, then either \(\RR \setminus A\) is bounded or \(\RR \setminus A\) is unbounded and not connected. We claim that in both cases at least one of the connected components of \(\RR \setminus A\) is bounded. 

The last claim is trivially valid if \(\RR \setminus A\) is bounded. Let \(\RR \setminus A\) be unbounded and not connected. Suppose that every connected component of \(\RR \setminus A\) is unbounded. Since every two different connected components are disjoint and each of them can be represented as \((-\infty, c)\) or \((c, \infty)\) for some \(c \in \RR\), the set \(\RR \setminus A\) contains exactly two components \((-\infty, c_1)\) and \((c_2, \infty)\) such that \(c_1 \leqslant c_2\). Hence, the equality \(A = [c_1, c_2]\) holds, that is impossible because \(A\), by condition~\ref{t4.3:c2}, is isometric to the unbounded subspace \(Y\) of the space \((X, d)\). The existence of bounded connected component of \(\RR \setminus A\) follows.

To complete the proof it suffices to show that \ref{t4.3:c2} is false whenever the set \(\RR \setminus A\) contains a bounded connected component. It is easy to show. Indeed, the set \(Com(\RR \setminus A)\) of all connected components is at most countable and the cardinality of \((0, \infty)\) is uncountably infinite. Consequently, if an open interval \((a, b)\), \(-\infty < a < b < \infty\), belongs to \(Com(\RR \setminus A)\), then there is \(k_1 \in (0, \infty)\) such that
\begin{equation}\label{t4.3:e7}
k_1^{-1} |b-a| \neq |c- d|
\end{equation}
for every bounded connected component \((c, d)\) of \(\RR \setminus A\). By condition~\ref{t4.3:c2}, \(A\) and \(k_1^{-1}A\) are isometric. It implies that \(\RR \setminus A\) and \(\RR \setminus k_1^{-1} A\) are isometric by Lemma~\ref{l4.4}. The open interval \(\left(k_1^{-1} a, k_1^{-1} b\right)\) is a connected component of \(\RR \setminus k_1^{-1} A\). Hence, the isometricity of \(\RR \setminus A\) and \(\RR \setminus k_1^{-1} A\) implies the existence of an open interval \((c, d)\) such that \((c, d)\) is a connected component of \(\RR \setminus A\) and the equality
\[
\left|k_1^{-1} a - k_1^{-1} b\right| = |c - d|
\]
holds. The last equality contradicts \eqref{t4.3:e7}. The proof is completed.
\end{proof}

\section{Asymptotically equivalent subspaces and asymptotically isometric spaces}
\label{sec7}

Let \((X,d)\) be a metric space, \(\tr = \seq{r_n}{n}\) be a scaling sequence and let \(Y\), \(Z\) be unbounded metric subspaces of \((X,d)\).

\begin{definition}\label{d7.1}
\(Y\) and \(Z\) are asymptotically equivalent with respect to \(\tr\) if there are mappings \(\bm{\Phi} \colon Seq(Y, \tr) \to Seq(Z, \tr)\) and \(\bm{\Psi} \colon Seq(Z, \tr) \to Seq(Y, \tr)\) such that 
\begin{equation}\label{d7.1:e1}
d^{\tr}(\ty, \bm{\Phi}(\ty)) = d^{\tr}(\tz, \bm{\Psi}(\tz)) = 0
\end{equation}
for every \(\ty \in Seq(Y, \tr)\) and every \(\tz \in Seq(Z, \tr)\), where \(d^{\tr}\) is a pseudometric on \(Seq(X, \tr)\) defined by~\eqref{e2.10}.
\end{definition}

\begin{remark}\label{r7.2}
It is easy to see that \(Y\) and \(Z\) are asymptotically equivalent with respect to \(\tr\) if and only if, for all \(\seq{y_n^1}{n} \in Seq(Y, \tr)\) and \(\seq{z_n^1}{n} \in Seq(Z, \tr)\), there exist \(\seq{y_n^2}{n} \in Seq(Y, \tr)\) and \(\seq{z_n^2}{n} \in Seq(Z, \tr)\) satisfying the limit relations
\[
\lim_{n \to \infty} \frac{d(y_n^1, z_n^2)}{r_n} = \lim_{n \to \infty} \frac{d(y_n^2, z_n^1)}{r_n} = 0.
\]
\end{remark}

The first our goal is to show that the mappings \(\bm{\Phi}\) and \(\bm{\Psi}\) satisfying \eqref{d7.1:e1} are pseudoisometries in the sense of Definition~\ref{d2.5}. The following lemma shows that these mappings are distance preserving. 

\begin{lemma}\label{l7.2}
Let \((X, d)\) be a metric space, \(\tr\) be a scaling sequence, \(Y\), \(Z\) be unbounded subspaces of \((X, d)\) and let \(\bm{\Phi} \colon Seq(Y, \tr) \to Seq(Z, \tr)\) satisfy
\begin{equation}\label{l7.2:e1}
d^{\tr}(\ty, \bm{\Phi}(\ty)) = 0
\end{equation}
for all \(\ty \in Seq(Y, \tr)\). Then the equality 
\begin{equation}\label{l7.2:e2}
d^{\tr}(\bm{\Phi}(\ty_1), \bm{\Phi}(\ty_2)) = d^{\tr}(\ty_1, \ty_2)
\end{equation}
holds for all \(\ty_1\), \(\ty_2 \in Seq(Y, \tr)\).
\end{lemma}

\begin{proof}
Let \(\ty_1\), \(\ty_2 \in Seq(Y, \tr)\). Then, using the triangle inequality and \eqref{l7.2:e1}, we obtain
\[
\left|d^{\tr}(\bm{\Phi}(\ty_1), \bm{\Phi}(\ty_2)) - d^{\tr}(\ty_1, \ty_2)\right| \leqslant \left|d^{\tr}(\ty_1, \bm{\Phi}(\ty_1))\right| + \left|d^{\tr}(\ty_2, \bm{\Phi}(\ty_2))\right| = 0.
\]
Equality~\eqref{l7.2:e2} follows.
\end{proof}

\begin{lemma}\label{l7.3}
Let \((X, d)\) be a metric space, \(\tr\) be a scaling sequence, let \(Y\) and \(Z\) be unbounded subspaces of \((X, d)\) which are asymptotically equivalent with respect to \(\tr\). If \(\bm{\Phi} \colon Seq(Y, \tr) \to Seq(Z, \tr)\) satisfy \eqref{l7.2:e1} for every \(\ty \in Seq(Y, \tr)\), then, for every \(\tz_1 \in Seq(Z, \tr)\), there is \(\ty_1 \in Seq(Y, \tr)\) such that
\begin{equation}\label{l7.3:e2}
d^{\tr}(\bm{\Phi}(\ty_1), \tz_1) = 0.
\end{equation}
\end{lemma}

\begin{proof}
Suppose that \(\bm{\Phi} \colon Seq(Y, \tr) \to Seq(Z, \tr)\) satisfies \eqref{l7.2:e1} for every \(\ty \in Seq(Y, \tr)\). Since \(Y\) and \(Z\) are asymptotically equivalent with respect to \(\tr\), there is \(\bm{\Psi} \colon Seq(Z, \tr) \to Seq(Y, \tr)\) such that
\begin{equation}\label{l7.3:e3}
d^{\tr}(\bm{\Psi}(\tz), \tz) = 0
\end{equation}
for every \(\tz \in Seq(Z, \tr)\).

Let \(\tz_1\) be an arbitrary element of \(Seq(Z, \tr)\). Write \(\ty_1 := \bm{\Psi}(\tz_1)\). Then \(\ty_1 \in Seq(Y, \tr)\) and, using \eqref{l7.2:e1} with \(\ty = \ty_1\), we obtain
\begin{equation}\label{l7.3:e4}
d^{\tr}(\bm{\Phi}(\ty_1), \ty_1) = d^{\tr}(\bm{\Phi}(\bm{\Psi}(\tz_1)), \ty_1) = 0.
\end{equation}
It follows from \eqref{l7.3:e3} with \(\tz = \tz_1\) that 
\begin{equation}\label{l7.3:e5}
d^{\tr}(\bm{\Psi}(\tz_1), \tz_1) = 0.
\end{equation}
Since \(\coherent{0}\) is an equivalence relation on every pseudometric space (see Proposition~\ref{ch2:p1}), equalities \eqref{l7.3:e4}--\eqref{l7.3:e5} imply \eqref{l7.3:e2}.
\end{proof}

\begin{proposition}\label{p7.4}
Let \((X, d)\) be a metric space, \(\tr\) be a scaling sequence, and let \(Y\) and \(Z\) be unbounded subspaces of \((X, d)\) which are asymptotically equivalent with respect to \(\tr\). Then the pseudometric subspaces \(Seq(Y, \tr)\) and \(Seq(Z, \tr)\) of the pseudometric space \((Seq(X, \tr), d^{\tr})\) are pseudoisometric.
\end{proposition}

\begin{proof}
By Definition~\ref{d7.1}, there are \(\bm{\Phi} \colon Seq(Y, \tr) \to Seq(Z, \tr)\) and \(\bm{\Psi} \colon Seq(Z, \tr) \to Seq(Y, \tr)\) such that
\[
d^{\tr}(\ty, \bm{\Phi}(\ty)) = d^{\tr}(\tz, \bm{\Psi}(\tz)) = 0
\]
holds for all \(\ty \in Seq(Y, \tr)\) and \(\tz \in Seq(Z, \tr)\). Hence, by Lemma~\ref{l7.2}, the mapping \(\bm{\Phi}\) is distance preserving. Moreover, by Lemma~\ref{l7.3}, for every \(\tz_1 \in Seq(Z, \tr)\) there is \(\ty_1 \in Seq(Y, \tr)\) such that
\[
d^{\tr}(\bm{\Phi}(\ty_1), \tz_1) = 0.
\]
Hence, \(\bm{\Phi}\) is a pseudoisometry of \(Seq(Y, \tr)\) and \(Seq(Z, \tr)\) by Definition~\ref{d2.5}.
\end{proof}

\begin{corollary}\label{c7.5}
Let \((X, d)\) be a metric space, \(\tr\) be a scaling sequence, and let \(Y\) and \(Z\) be unbounded subspaces of \((X, d)\) which are asymptotically equivalent with respect to \(\tr\). Then the metric identification \(\bfpretan{Y}{r}\) of \(Seq(Y, \tr)\) is isometric to the metric identification \(\bfpretan{Z}{r}\) of \(Seq(Z, \tr)\).
\end{corollary}

\begin{proof}
It follows directly from Proposition~\ref{p7.4} and Theorem~\ref{t2.10}.
\end{proof}

Let \(A\) be an unbounded subset of a metric space \((X, d)\) and let \(\tr\) be a scaling sequence. In what follows, we will denote by \(\overline{Seq(A, \tr)}\) the closure of \(Seq(A, \tr)\) in the pseudometric space \((Seq(X, \tr), d^{\tr})\).

\begin{proposition}\label{p7.2}
Let \((X, d)\) be an unbounded metric space and let \(\tr\) be a scaling sequence. Then the equality
\begin{equation}\label{p7.2:e1}
\overline{Seq(Y, \tr)} = \overline{Seq(Z, \tr)}
\end{equation}
holds for any two unbounded \(Y\), \(Z \subseteq X\) which are asymptotically equivalent with respect to \(\tr\).
\end{proposition}

\begin{proof}
It follows from Corollary~\ref{c2.8} and Definition~\ref{d7.1}.
\end{proof}

If Conjecture~\ref{con5.1} is valid, then, using Corollary~\ref{c2.9}, we obtain the following. 

\begin{proposition}\label{p7.7}
Let \((X, d)\) be a metric space and let \(\tr\) be a scaling sequence. Unbounded subspaces \(Y\) and \(Z\) of \((X, d)\) are asymptotically equivalent with respect to \(\tr\) if and only if the equality
\[
\overline{Seq(Y, \tr)} = \overline{Seq(Z, \tr)}
\]
holds.
\end{proposition}

We shall say that \(Y\) and \(Z\) are \emph{strongly asymptotically equivalent} if \(Y\) and \(Z\) are asymptotically equivalent for all scaling sequences \(\tr\). A sufficient condition under which \(Y\) and \(Z\) are strongly asymptotically equivalent can be given in terms of the so-called \(\varepsilon\)-nets.

Let \((X, d)\) be a metric space. The closed ball with center \(c \in X\) and radius \(r \in (0, \infty)\) is the set
\[
\overline{B}(c, r) := \{x \in X \colon d(x, c) \leqslant r\}.
\]
Let \(W\) be a subset of \(X\) and let \(\varepsilon > 0\). Then a set \(C \subseteq X\) is an \(\varepsilon\)-\emph{net} for \(W\) if
\[
W \subseteq \bigcup_{c \in C} \overline{B}(c, \varepsilon).
\]

\begin{proposition}\label{p5.2}
Let \(Y\) and \(Z\) be unbounded subspaces of a metric space \((X, d)\). If there is \(\varepsilon > 0\) such that \(Y\) is an \(\varepsilon\)-net to \(Z\) and vice versa, then \(X\) and \(Y\) are strongly asymptotically equivalent.
\end{proposition}

\begin{proof}
It follows directly from Remark~\ref{r7.2}.
\end{proof}

\begin{example}\label{ex5.3}
Let \(n \in \NN\) and let \((X, d)\) be the \(n\)-dimensional Euclidean space \(\RR^{n}\). Write \(\mathbb{Z}^n\) and \(\mathbb{Q}^n\) for the \(n\)-ary Cartesian powers of the set of all integer numbers and the set of all rational numbers, respectively. Then \(\mathbb{Z}^n\) and \(\mathbb{Q}^n\) are strongly asymptotically equivalent as subspaces of \(\RR^{n}\).
\end{example}

\begin{example}\label{ex3.1.2}
Let $Y$ be an unbounded subspace of $(X, d)$ and let $\ol{Y}$ be the closure of $Y$ in $(X,d)$. Then $Y$ and $\ol{Y}$ are strongly asymptotically equivalent.
\end{example}

Let $Y$ be an unbounded subspace of $(X,d)$, let $\tr$ be a scaling sequence and let $\tilde{F} \subseteq Seq(X, \tr)$. Let us define a set $[\tilde{F}]_Y$ by the rule:
\begin{equation}\label{e3.1.1}
(\ty \in [\tilde{F}]_Y) \Leftrightarrow \bigl((\ty \in Seq(Y, \tr)) \mathbin{\&} (\exists \tx \in \tilde{F} \colon d^{\tr}(\tx, \ty) = 0)\bigr).
\end{equation}
Note that $[\tilde{F}]_Y$ can be empty for a non-void $\tilde{F}$ if the set $X \setminus Y$ is ``big enough''.

\begin{proposition}\label{p3.1.2}
Let $Y$ and $Z$ be unbounded subspaces of an unbounded metric space $(X,d)$ and let $\tr$ be a scaling sequence. Suppose that $Y$ and $Z$ are asymptotically equivalent with respect to $\tr$. Then following statements hold for every maximal self-stable set $\sstable{Z}{r} \subseteq Seq(Z, \tr)$.
\begin{enumerate}
\item\label{p3.1.2s1} The family $[\sstable{Z}{r}]_Y$ is maximal self-stable and we have the equalities
\begin{equation}\label{p3.1.2e1}
[[\sstable{Z}{r}]_Y]_Z = \sstable{Z}{r} = [\sstable{Z}{r}]_Z.
\end{equation}
\item\label{p3.1.2s2} If $\pretan{Z}{r}$ and $\pretan{Y}{r}$ are the metric identifications of $\sstable{Z}{r}$ and, respectively, of $\sstable{Y}{r} := [\sstable{Z}{r}]_Y$, then the mapping 
\begin{equation}\label{p3.1.2e2}
\pretan{Z}{r} \ni \nu \longmapsto [\nu]_Y \in \pretan{Y}{r}
\end{equation}
is an isometry. Furthermore, if $\pretan{Z}{r}$ is tangent, then $\pretan{Y}{r}$ is also tangent.
\end{enumerate}
\end{proposition}

\begin{proof}
Suppose $\sstable{Z}{r}$ is a maximal self-stable subset of $Seq(Z, \tr)$.

\ref{p3.1.2s1} Let $\ty_1$, $\ty_2 \in \sstable{Y}{r}^* := [\sstable{Z}{r}]_{Y}$. Then, by~\eqref{e3.1.1}, there exist $\tz_1$, $\tz_2 \in \sstable{Z}{r}$ such that
\begin{equation}\label{p3.1.2e3}
d^{\tr}(\ty_1, \tz_1) = d^{\tr}(\ty_2, \tz_2) = 0.
\end{equation}
Since $\tz_1$ and $\tz_2$ are mutually stable, $\ty_1$ and $\ty_2$ are also mutually stable. Consequently, $\sstable{Y}{r}^*$ is self-stable. The similar arguments show that $[\sstable{Y}{r}^*]_{Z}$ is also self-stable. Moreover, since
$$
[[\sstable{Z}{r}]_{Y}]_{Z} = [\sstable{Y}{r}^*]_{Z} \supseteq \sstable{Z}{r},
$$
the maximality of $\sstable{Z}{r}$ implies the first equality in~\eqref{p3.1.2e1}. The second one also simple follows from the maximality of $\sstable{Z}{r}$. It still remains to prove that $\sstable{Y}{r}^*$ is a maximal self-stable subset of $Seq(Y, \tr)$. 

Let $\sstable{Y}{r}$ be a maximal self-stable set in $Seq(Y, \tr)$ such that $\sstable{Y}{r} \supseteq \sstable{Y}{r}^*$. Then $[\sstable{Y}{r}]_Z$ is self-stable and $[\sstable{Y}{r}]_Z \supseteq \sstable{Z}{r}$. Since $\sstable{Z}{r}$ is maximal self-stable, the last inclusion implies the equality $[\sstable{Y}{r}]_Z = \sstable{Z}{r}$. Using the last equality and~\eqref{p3.1.2e1} we obtain
$$
\sstable{Y}{r} = [[\sstable{Y}{r}]_Z]_Y = [\sstable{Z}{r}]_Y := \sstable{Y}{r}^*,
$$
i.e., $\sstable{Y}{r}^*$ is maximal self-stable.

\ref{p3.1.2s2} Let $\nu \in \pretan{Z}{r}$ and let $\tz \in \nu$. It follows from~\eqref{e3.1.1} that
\begin{equation}\label{p3.1.2e4}
[\nu]_Y = \{\ty \in Seq(Y, \tr)\colon d^{\tr}(\ty, \tz) = 0\}.
\end{equation}
The last equality implies that function~\eqref{p3.1.2e2} is distance-preserving. In addition, using~\eqref{p3.1.2e1}, we obtain
$$
[[\nu]_Y]_Z = \nu \text{ and } [[\beta]_Z]_Y = \beta
$$
for every $\nu \in \pretan{Z}{r}$ and every $\beta \in \pretan{Y}{r}$. Consequently function~\eqref{p3.1.2e2} is bijective. To prove that $\pretan{Y}{r}$ is tangent if $\pretan{Z}{r}$ is tangent we can use statement~\ref{p3.1.2s1} of the present proposition and statement~\ref{p1.3.1s2} of Proposition~\ref{p1.3.1}.
\end{proof}

The transition from a subspace to another subspace preserves the uniqueness of pretangent spaces if these subspaces are asymptotically equivalent.

\begin{corollary}\label{c3.1.5}
Let $Y$ and $Z$ be unbounded subspace of a metric space $X$. Suppose $Y$ and $Z$ are asymptotically equivalent with respect to a scaling sequence $\tr$ and there exists a unique maximal self-stable set $\sstable{Z}{r} \subseteq Seq(Z, \tr)$. Then $\sstable{Y}{r} := [\sstable{Z}{r}]_Y$ is a unique maximal self-stable subset of $Seq(Y, \tr)$.
\end{corollary}

\begin{proof}
Let $\sstable{Y}{r}^*$ be a maximal self-stable subset of $Seq(Y, \tr)$. Then, by statement~\ref{p3.1.2s1} of Proposition~\ref{p3.1.2}, $[\sstable{Y}{r}^*]_Z$ is a maximal self-stable subset of $Seq(Z, \tr)$. Hence, by~\eqref{p3.1.2e1},
$$
\sstable{Y}{r}^* = [[\sstable{Y}{r}^*]_Z]_Y = [\sstable{Z}{r}]_Y = \sstable{Y}{r}.
$$
The uniqueness of $\sstable{Y}{r}$ follows.
\end{proof}

\begin{example}\label{ex5.7}
Let \((X, d)\) be the complex plane \(\mathbb{C}\) with the usual metric. Write \(Z\) for the intersection of the strip \(\{z \in \mathbb{C} \colon c_1 \leqslant \operatorname{Im} z \leqslant c_2\}\), where \(c_1\) and \(c_2\) are the real numbers such that \(c_1 < 0 < c_2\), with the half-plane \(\{z \in \mathbb{C} \colon \operatorname{Re} z \geqslant 0\}\) and let 
\[
Y = \{z \in \mathbb{C} \colon \operatorname{Re} z \geqslant 0 \text{ and } \operatorname{Im} z = 0\}.
\]
Then, by Proposition~\ref{p5.2}, \(Y\) and \(Z\) are strongly asymptotically equivalent. By Corollary~\ref{c4.3}, for every scaling sequence \(\tr\), \(Y\) has the unique pretangent space \(\pretan{Y}{r}\) and, by Statement~\ref{p1.3.2:s3} of Proposition~\ref{p1.3.2}, \(\pretan{Y}{r}\) is isometric to \(\RR_{+}\) and tangent. Now, using Proposition~\ref{p1.3.2} and Corollary~\ref{c3.1.5}, we obtain that \(Z\) also has the unique pretangent space \(\pretan{Z}{r}\) and this space is tangent and isometric to \(\RR_{+}\) for every \(\tr\).
\end{example}

Let $(X,d)$ be a metric space and let $p \in X$. For every $t > 0$ we denote by $S(p,t)$ the sphere with the radius $t$ and the center $p$,
$$
S(p,t) := \{x \in X\colon d(x,p) = t\},
$$
and for every $Y \subseteq X$ we write
$$
S_t^Y := S(p,t) \cap Y.
$$
Let $Y$ and $Z$ be subspaces of $(X,d)$. Define
\begin{equation}\label{e3.1.6}
\varepsilon(t,Z,Y) := \sup_{z \in S_t^Z} \inf_{y \in Y} d(z,y)
\end{equation}
and
\begin{equation}\label{e3.1.7}
\varepsilon(t) = \max\{\varepsilon(t,Z,Y), \varepsilon(t,Y,Z)\},
\end{equation}
where we set $\varepsilon(t,Z,Y) = 0$ if $S_t^Z = \varnothing$ and, respectively, $\varepsilon(t,Y,Z) = 0$ if $S_t^Y = \varnothing$.

\begin{remark}\label{r3.1.6}
The quantity $\varepsilon(t)$ is a special case of the following \emph{conditional Hausdorff distance}. Let $(X,d)$ be a metric space and let $A \subseteq Y \subseteq X$ and $B \subseteq Z \subseteq X$. We define the conditional Hausdorff distance between $A$ and $B$ (with respect to $Y$ and $Z$) by
\begin{equation}\label{r3.1.6e1}
d_H (A, B| Y,Z) := \max\{\sup_{z \in B}\inf_{y \in Y} d(z,y), \sup_{y \in A}\inf_{z \in Z} d(z,y)\}.
\end{equation}
It is clear that we have
$$
d_H (A, B| Y,Z) = \varepsilon(t)
$$
for $A = S(p,t) \cap Y$ and $B = S(p,t) \cap Z$ and, moreover, 
$$
d_H (A, B| Y,Z) = d_H (A, B)
$$
for $A = Y$ and $B = Z$, where $d_H (A, B)$ is the usual Hausdorff distance between $A$ and $B$. (See, for example, book~\cite{BBI} for the properties of the Hausdorff distance and connected them properties of the Gromov--Hausdorff convergence.)
\end{remark}

\begin{theorem}\label{t3.1.6}
Let $Y$ and $Z$ be unbounded subspaces of a metric space $(X,d)$. Then $Y$ and $Z$ are strongly asymptotically equivalent if and only if
\begin{equation}\label{t3.1.6e1}
\lim_{t\to \infty} \frac{\varepsilon(t)}{t} = 0.
\end{equation}
\end{theorem}
\begin{proof}
Suppose that limit relation~\eqref{t3.1.6e1} holds. Let $\tr = \seq{r_n}{n}$ be a scaling sequence, let $p \in X$, and let $\tz = \seq{z_n}{n} \in Seq(Z, \tr)$. To find $\ty = \seq{y_n}{n} \in Seq(Y, \tr)$ such that 
\begin{equation}\label{t3.1.6e2}
d^{\tr}(\ty, \tz) = 0
\end{equation}
note that the statement $\tz \in Seq(Z, \tr)$ implies
\begin{equation}\label{t3.1.6e3}
\lim_{n\to\infty} d(z_n, p) = \infty.
\end{equation}
It follows from~\eqref{e3.1.7} and~\eqref{t3.1.6e1} that
\begin{equation}\label{t3.1.6e4}
\lim_{t\to\infty} \frac{\varepsilon(t,Z,Y)}{t} = 0.
\end{equation}
By~\eqref{t3.1.6e3} there is $n_0 \in \NN$ such that $d(z_n, p) > 0$ if $n \geq n_0$. Write
\begin{equation}\label{t3.1.6e5}
t_n := \begin{cases}
1 & \text{if $n < n_0$}\\
d(z_n, p) & \text{if $n \geq n_0$}.
\end{cases}
\end{equation}
Let $c>0$. The definition of $\varepsilon(t,Z,Y)$ and~\eqref{t3.1.6e4} imply, for all sufficiently large $n \in \NN$, there are $y_n \in Y$ with
\begin{equation}\label{t3.1.6e6}
d(z_n, y_n) \leq \varepsilon(t_n,Z,Y) + c t_n.
\end{equation}
Set $\ty := \seq{y_n}{n}$, where $y_n$ are some points of $Y$ for which~\eqref{t3.1.6e6} holds. Now using~\eqref{t3.1.6e3}--\eqref{t3.1.6e5} we obtain
\begin{multline*}
d^{\tr}(\tz, \ty) = \limsup_{n\to\infty} \frac{d(z_n, y_n)}{r_n} \leq \lim_{n\to\infty} \frac{d(z_n, p)}{r_n} \limsup_{n\to\infty} \frac{d(z_n, y_n)}{t_n} \\*
\leq \Dist{z} \limsup_{n\to\infty} \frac{\varepsilon(t_n,Z,Y) + ct_n}{t_n} = c.
\end{multline*}
Letting $c$ to zero, we obtain \eqref{t3.1.6e2}. Similarly, we can prove that, for every $\ty \in Seq(Y, \tr)$, there is $\tz \in Seq(Z, \tr)$ such that $d^{\tr}(\tz, \ty) = 0$. Hence if~\eqref{t3.1.6e1} holds, then $Y$ and $Z$ are strongly asymptotically equivalent.

Suppose now that~\eqref{t3.1.6e1} does not hold. Without loss of generality we can assume that 
$$
\limsup_{t\to \infty} \frac{\varepsilon(t,Z,Y)}{t} > 0.
$$
Then there is a sequence $\tilde{t}$ of positive numbers $t_n$ with $\lim_{n\to\infty} t_n = \infty$ and there is a constant $c > 0$ such that, for every $n \in \NN$, there exists $z_n \in S_{t_n}^Z$ for which
\begin{equation}\label{t3.1.6e7}
\inf_{y \in Y} d(z_n,y) \geq ct_n = cd(p, z_n).
\end{equation}
Let us denote by $\tz$ the sequence of points $z_n \in S_{t_n}^Z$ which satisfy~\eqref{t3.1.6e7}. Take the sequence $\tilde{t} = \seq{t_n}{n}$ as a scaling sequence. Then $\tz \in Seq(Z, \tilde{t})$ and, by~\eqref{t3.1.6e7},
$$
d^{\tr}(\tz, \ty) \geq c > 0
$$ 
holds for every $\ty = \seq{y_n}{n} \in Seq(Y, \tilde{t})$. Consequently, $Y$ and $Z$ are not strongly asymptotically equivalent.
\end{proof}

Let $(X, d)$ be an unbounded metric space and let $Y$ be an unbounded metric subspace of $X$. If $\sstable{X}{r} \subseteq Seq(X, \tr)$ and $\sstable{Y}{r} \subseteq Seq(Y, \tr)$ are maximal self-stable sets and $\sstable{Y}{r} \subseteq \sstable{X}{r}$, then there is a unique isometric embedding $Em_Y \colon \pretan{Y}{r} \to \pretan{X}{r}$ such that the following diagram
\begin{equation}\label{t3.1.6e8}
\ctdiagram{
\def\x{40}
\ctv -\x,\x: {\sstable{Y}{r}}
\ctv -\x,-\x: {\pretan{Y}{r}}
\ctv \x,\x: {\sstable{X}{r}}
\ctv \x,-\x: {\pretan{X}{r}}
\ctet -\x,\x,\x,\x:{in}
\ctet -\x,-\x,\x,-\x:{Em_Y}
\ctel -\x,\x,-\x,-\x:{\pi_Y}
\cter \x,\x,\x,-\x:{\pi_X}
}
\end{equation}
is commutative. Here $\pretan{X}{r}$ and $\pretan{Y}{r}$ are the pretangent spaces corresponding to $\sstable{X}{r}$ and, respectively, to $\sstable{Y}{r}$, $\pi_Y$ and $\pi_X$ are the appropriate natural projections defined similarly to \eqref{e1.1.9} and $in(\ty) = \ty$ for all $\ty \in \sstable{Y}{r}$.
\medskip

Using Theorem~\ref{t3.1.6} with $Z = X$ we obtain the following corollary.

\begin{corollary}\label{c3.1.8}
Let $(X,d)$ be an unbounded metric space and let $Y$ be an unbounded subspace of $X$. Then the following conditions are equivalent.
\begin{enumerate}
\item\label{c3.1.8s1} For every $\tr$ and every maximal self-stable $\sstable{X}{r} \subseteq Seq(X, \tr)$ there is a maximal self-stable $\sstable{Y}{r} \subseteq Seq(Y, \tr)$ such that $\sstable{Y}{r} \subseteq \sstable{X}{r}$ and the embedding $E_{m_Y} \colon \pretan{Y}{r} \to \pretan{X}{r}$ is an isometry.
\item\label{c3.1.8s2} The equality
$$
\lim_{t\to\infty} \frac{\varepsilon(t, X, Y)}{t} = 0
$$
holds.
\item\label{c3.1.8s3} $X$ and $Y$ are strongly asymptotically equivalent.
\end{enumerate}
\end{corollary}

\begin{proof}
The equivalence \ref{c3.1.8s2} $\Leftrightarrow$ \ref{c3.1.8s3} follows from Theorem~\ref{t3.1.6}. To prove \ref{c3.1.8s3} $\Rightarrow$ \ref{c3.1.8s1} we note that the equality
$$
[\sstable{X}{r}]_Y = \tilde{Y} \cap \sstable{X}{r}
$$
holds for every maximal self-stable $\sstable{X}{r}$ if $X$ and $Y$ are strongly asymptotically equivalent. Consequently, \ref{c3.1.8s3} implies \ref{c3.1.8s1} because the mapping $E_{m_Y}$ is an isometry whenever $X$ and $Y$ are strongly asymptotically equivalent.
\end{proof}

\begin{remark}
Theorem~\ref{t3.1.6} and Corollary~\ref{c3.1.8} can be considered as asymptotic modifications of previously proved facts from~\cite{Dov2008}.
\end{remark}

In the rest of the section we expand the concept of asymptotically equivalent subspaces of metric spaces to the concept of \emph{asymptotically isometric} metric spaces.

For every mapping \(F \colon A \to B\) we will denote by \(\tldbf{F}\) the mapping defined on the set of all sequences \(\ta = \seq{a_n}{n} \subseteq A\) such that \(\tldbf{F}(\tld{a}) := \tld{b} = \seq{b_n}{n} \subseteq B\) with \(b_n = F(a_n)\) for every \(n \in \NN\). Moreover, if \(\tld{A}\) is a set of sequences whose elements belong to \(A\), then we write
\begin{equation}\label{e7.25}
\tldbf{F}(\tld{A}) := \{\tldbf{F}(\ta) \colon \ta \in \tld{A}\}.
\end{equation}
Now we are ready to introduce the concept of asymptotically isometric spaces.

\begin{definition}\label{d7.15}
Let \((Y, \rho)\) and \((Z, \delta)\) be unbounded metric spaces and let $\tr = \seq{r_n}{n}$ be a scaling sequence. The spaces \((Y, \rho)\) and \((Z, \delta)\) are \emph{asymptotically isometric} with respect to $\tr$ if there are mappings \(\Phi \colon Y \to Z\) and \(\Psi \colon Z \to Y\) such that:
\begin{enumerate}
\item \label{d7.15:s1} \(\tldbf{\Phi}(\ty)\) belongs to \(Seq(Z, \tr)\) for every \(\tld{y} \in Seq(Y, \tr)\);
\item \label{d7.15:s2} \(\tldbf{\Psi}(\tz)\) belongs to \(Seq(Y, \tr)\) for every \(\tld{z} \in Seq(Z, \tr)\),
\end{enumerate}
where \(\tldbf{\Phi}\) and \(\tldbf{\Psi}\) are defined as in \eqref{e7.25} and, in addition, the restrictions \(\tldbf{\Phi}|_{Seq(Y, \tr)}\) and \(\tldbf{\Psi}|_{Seq(Z, \tr)}\) are pseudoisometries.
\end{definition}

The mappings \(\Phi\) and \(\Psi\) satisfying Definition~\ref{d7.15} will be called \emph{asymptotic isometries} of metric spaces \(Y\) and \(Z\). Moreover, we will say that $Y$ and $Z$ are \emph{strongly asymptotically isometric} if $Y$ and $Z$ are asymptotically isometric for all scaling sequences.

\begin{example}\label{ex7.2}
Let \(Y\) and \(Z\) be isometric metric spaces. Then \(Y\) and \(Z\) are strongly asymptotically isometric and every isometry \(\Phi \colon Y \to Z\) is an asymptotic isometry for every scling sequence \(\tr\).
\end{example}

Using Theorem~\ref{t2.10} we obtain.

\begin{proposition}\label{p7.21}
Let \(X_1\) and \(X_2\) be unbounded metric spaces and \(\tr = \seq{r_n}{n}\) be a scaling sequence. If \(X_1\) and \(X_2\) are asymptotically isometric with respect to \(\tr\), then the metric spaces \((\bfpretan{X_1}{r}, \bm{\rho}_1)\) and \((\bfpretan{X_2}{r}, \bm{\rho}_2)\) are isometric.
\end{proposition}

The following theorem shows, in particular, that the concept of asymptotically isometric metric spaces is an extension of the concept of asymptotically equivalent subspaces of a given metric space.

\begin{theorem}\label{t7.22}
Let \((X, d)\) be a metric space, \(\tr = \seq{r_n}{n}\) be a scaling sequence and let \(Y\), \(Z\) be unbounded and asymptotically equivalent with respect to \(\tr\) subspaces of \((X, d)\). Then \(Y\) and \(Z\) are asymptotically isometric with respect to \(\tr\).
\end{theorem}

In the proof of Theorem~\ref{t7.22} we will use the concept of distance from a point to set.

Let \((X, d)\) be a metric space. For every \(x \in X\) and every nonempty set \(A \subseteq X\), the \emph{distance} from \(x\) to \(A\) is defined by
\[
\Delta (x, A) := \inf_{y \in A} d(x, y).
\]

\begin{proof}[Proof of Theorem~\(\ref{t7.22}\)]
Let \(\varepsilon_1 > 0\) be fixed. Then, for every \(x \in X\) and every nonempty \(A \subseteq X\), there is \(a_x \in A\) such that
\begin{equation}\label{t7.22:e1}
|d(x, a_x) - \Delta (x, A)| \leqslant \varepsilon_1.
\end{equation}
Consequently, there is a mapping \(\Phi \colon Y \to Z\) such that
\begin{equation}\label{t7.22:e2}
|d(\Phi(y), y) - \Delta (y, Z)| \leqslant \varepsilon_1
\end{equation}
for every \(y \in Y\). Similarly, there exists \(\Psi \colon Z \to Y\) satisfying the inequality
\begin{equation}\label{t7.22:e3}
|d(\Psi(z), z) - \Delta (z, Y)| \leqslant \varepsilon_1
\end{equation}
for every \(z \in Z\). We claim that \(\Phi \colon Y \to Z\) and \(\Psi \colon Z \to Y\) are asymptotic isometries. Let us prove it.

Since \(Y\) and \(Z\) are asymptotically equivalent with respect to \(\tr\), there is a mapping \(\tldbf{\Phi} \colon Seq(Y, \tr) \to Seq(Z, \tr)\) such that 
\begin{equation}\label{t7.22:e4}
d^{\tr}(\ty, \tldbf{\Phi}(\ty)) = 0
\end{equation}
for every \(\ty \in Seq(Y, \tr)\).

Let \(\ty = \seq{y_n}{n}\) belong to \(Seq(Y, \tr)\). Then \(\tldbf{\Phi}(\ty) \in Seq(Z, \tr)\) can be written as \(\tldbf{\Phi}(\ty) = \seq{z_{n, \ty}}{n}\) and, consequently, we may rewrite \eqref{t7.22:e4} in the form
\begin{equation}\label{t7.22:e5}
\limsup_{n \to \infty} \frac{d(y_n, z_{n, \ty})}{r_n} = 0.
\end{equation}
The definition of the distance from \(y \in Y\) to \(Z\) implies that
\begin{equation}\label{t7.22:e7}
\Delta(y_n, Z) \leqslant d(y_n, z_{n, \ty})
\end{equation}
because \(z_{n, \ty} \in Z\) for every \(n \in \NN\). Inequalities~\eqref{t7.22:e7} and \eqref{t7.22:e2} imply
\[
d(\Phi(y_n), y_n) \leqslant \varepsilon_1 + \Delta(y_n, Z) \leqslant \varepsilon_1 + d(y_n, z_{n, \ty}).
\]
Now using \eqref{t7.22:e5} we have
\begin{equation}\label{t7.22:e6}
\limsup_{n \to \infty} \frac{d(\Phi(y_n), y_n)}{r_n} = 0.
\end{equation}
Consequently, by Lemma~\ref{l4.6}, the sequence \(\seq{\Phi(y_n)}{n}\) belongs to \(Seq(Z, \tr)\) for every \(\seq{y_n}{n} \in Seq(Y, \tr)\). Hence, the mapping \(\tldbf{\Phi}\) satisfies condition~\ref{d7.15:s1} of Definition~\ref{d7.15}.

Repeating the above arguments, we find that the mapping \(\tldbf{\Psi}\) satisfies condition~\ref{d7.15:s2} of Definition~\ref{d7.15} and that the equality
\[
d^{\tr}(\tz, \tldbf{\Psi}(\tz)) = 0
\]
holds for every \(\tz \in Seq(Z, \tr)\).

Now as in the proof of Proposition~\ref{p7.4}, we see that
\[
\tldbf{\Phi} \colon Seq(Y, \tr) \to Seq(Z, \tr) \quad \text{and} \quad \tldbf{\Psi} \colon Seq(Z, \tr) \to Seq(Y, \tr)
\]
are pseudoisometries. Thus, \(Y\) and \(Z\) are asymptotically isometric with respect to \(\tr\) by definition. 
\end{proof}

We conclude this section by the following conjecture describing the interconnections between asymptotically isometric metric spaces and asymptotically equivalent metric subspaces of a given metric space.

\begin{conjecture}\label{con7.22}
Let \(Y\) and \(Z\) be unbounded metric spaces and let \(\tr\) be a scaling sequence. Then the following conditions are equivalent.
\begin{enumerate}
\item \label{con7.22:s1} \(Y\) and \(Z\) are asymptotically isometric with respect to \(\tr\).
\item \label{con7.22:s2} There are a metric space \(X\) and unbounded \(Y_1\), \(Z_1 \subseteq X\) such that:
\begin{enumerate}
\item \(Seq(Y, \tr)\) is pseudoisometric to \(Seq(Y_1, \tr)\);
\item \(Seq(Z, \tr)\) is pseudoisometric to \(Seq(Z_1, \tr)\);
\item \(Y_1\) and \(Z_1\) are asymptotically equivalent with respect to \(\tr\).
\end{enumerate}
\end{enumerate}
\end{conjecture}


\providecommand{\bysame}{\leavevmode\hbox to3em{\hrulefill}\thinspace}
\providecommand{\MR}{\relax\ifhmode\unskip\space\fi MR }
\providecommand{\MRhref}[2]{%
  \href{http://www.ams.org/mathscinet-getitem?mr=#1}{#2}
}
\providecommand{\href}[2]{#2}

\end{document}